\documentclass[a4paper,12pt,reqno]{amsart}
\usepackage{latexsym, amsfonts, amsmath, amsthm, amssymb, enumerate, mathtools, tikz-cd, verbatim}
\usepackage[top=3cm, bottom=3cm, left=2.54cm, right=2.54cm]{geometry}

\usepackage{todonotes}

\renewcommand\Re{\mathop{{\rm Re}}}

\def\colvec[#1,#2]{\begin{bmatrix} #1 \\ #2 \end{bmatrix}}
\def\rowvec[#1,#2]{\begin{bmatrix} #1 & #2 \end{bmatrix}}
\def\ip<#1,#2>{\left\langle #1,#2 \right\rangle}

\newcommand{\trace}{\mathop{\mathrm{trace}}}
\newtheorem{thm}{Theorem}[section]
\newtheorem*{thm*}{Theorem}
\newtheorem{cor}[thm]{Corollary}
\newtheorem*{cor*}{Corollary}
\newtheorem{lem}[thm]{Lemma}
\newtheorem{prop}[thm]{Proposition}

\newtheorem*{con*}{Conjecture}
\newtheorem*{prob*}{Problem}

\theoremstyle{definition}
\newtheorem{defn}[thm]{Definition}
\newtheorem{rem}[thm]{Remark}

\newtheorem{ex}[thm]{Example}

\numberwithin{equation}{section}

\title[Hasimoto frames and the Gibbs measure of periodic NLSE]{Hasimoto frames and the Gibbs measure of periodic nonlinear Schr\"odinger equation}
\author{Gordon Blower, Azadeh Khaleghi, Moe Kuchemann-Scales}
\address{Department of Mathematics and Statistics, Lancaster University, United Kingdom.}
\date{13th May 2022}

\begin{document}

\maketitle
{\bf Abstract} The paper interprets the cubic nonlinear Schr\"odinger equation as a Hamiltonian system with infinite dimensional phase space. There is a Gibbs measure which is invariant under the flow associated with the canonical equations of motion. The logarithmic Sobolev and concentration of measure inequalities hold for the Gibbs measures, and here are extended to the $k$-point correlation function and distributions of related empirical measures. By Hasimoto's theorem, NLSE gives a Lax pair of coupled ODE for which the solutions give a system of moving frames. The paper studies the evolution of the measure induced on the moving frames by the Gibbs measure. \par
\vskip.1in

 \par
\vskip.1in

\section{Introduction}

Consider the Hamiltonian
\begin{equation}\label{Hamiltonian}H_3={\frac{1}{2}}\int_{\mathbb T} \Bigl( \Bigl( {\frac{\partial P}{\partial x}}\Bigr)^2+\Bigl( {\frac{\partial Q}{\partial x}}\Bigr)^2\Bigr)dx+ {\frac{\beta}{2\gamma}}\int_{\mathbb T} \bigl( P^2+Q^2\bigr)^\gamma dx,\end{equation}
on ${\mathbb T}={\mathbb R}/2\pi {\mathbb Z}$ which gives the canonical equations of motion
\begin{equation}\label{canonical}\begin{bmatrix} 0&1\\ -1&0\end{bmatrix} {\frac{\partial}{\partial t}} \begin{bmatrix} Q\\ P\end{bmatrix} =-{\frac{\partial^2}{\partial x^2}}\begin{bmatrix} Q\\ P\end{bmatrix} +\beta (P^2+Q^2)^{\gamma -1}
\begin{bmatrix} Q\\ P\end{bmatrix}\end{equation} 
so $u=P+iQ$ satisfies the nonlinear Schr\"odinger equation
\begin{equation}\label{NLS}i{\frac{\partial u}{\partial t}} =-{\frac{\partial^2u}{\partial x^2}} +\beta \vert u\vert^{2(\gamma -1)} u.\end{equation}
When $\gamma=2$, we have the cubic nonlinear Schr\"odinger equation. The spatial variable is $x\in {\mathbb T}$, and the functions are periodic, so that the system applies to fields parametrized by a circle.\par
\indent Throughout the paper, we write $(M,d,\mu )$ for a complete and separable metric space with a Radon (inner regular) probability measure $\mu$ on the $\sigma$-algebra generated by the Borel subsets.\par
\indent The squared $L^2$ norm $H_1=\int (P^2+Q^2)$ is formally invariant under the canonical equations of motion, so we can consider possible invariant measures on 
\begin{equation}\label{ball}B_K=\Bigl\{ u=P+iQ: P,Q\in L^2({\mathbb T}; {\mathbb R}): \int_{\mathbb T}(P^2+Q^2)dx\leq K\Bigr\}.\end{equation}
The Gibbs measure on $B_K$ for this micro-canonical ensemble is
\begin{equation}\label{Gibbs}\mu_{K, \beta} (du)=Z_K(\beta )^{-1}{\mathbb I}_{B_K}(u)\exp\Bigl( -{\frac{\beta}{4}}\int_{\mathbb T} \vert u\vert^4dx\Bigr) W(du)\end{equation}
where $W(du)$ is Wiener loop measure, and $Z_K(\beta )$ is a normalizing constant. Lebowitz, Rose and Speer \cite{LRS} proved existence of such an invariant measure, so that for all $K>0$ and $\beta\in {\mathbb R}$ there exists $Z_K(\beta )>0$ such that $\mu_{K,\beta}$ is a Radon probability measure on $B_K\subset L^2({\mathbb T}; {\mathbb R}^2)$. When $\beta=0$, we refer to the measure as free Wiener loop measure, indicating that the dynamics are free of potentials. For $\beta<0$, (\ref{NLS}) is said to be focussing and the Hamiltonian is unbounded below, giving the source of the technical problem.\par
\indent Bourgain \cite{Bo} gave an alternative existence proof using random Fourier series, and showed that the measure is invariant under the flow in the sense that the Cauchy problem is well posed on the support. Further refinements include a result of McKean \cite{M}, that the sample paths are H\"older continuous, and a result from Theorem 1.2(iv) in \cite{BBD} that the invariant measure of the micro-canonical ensemble satisfies a logarithmic Sobolev inequality. Random Fourier series fit naturally into Sturm's theory of metric measure spaces, which we use to reduce some of the analysis to invariant measures on finite-dimensional Hamiltonian systems.\par 
\indent The focusing case for spatial variable $x\in {\mathbb R}$ captures soliton solutions, and \cite{LRS} discuss the possible transition of the system between an ambient bounded random field and a soliton solution. For $x\in {\mathbb T}$, the notion of a spatially localized solution is inapplicable, but some of the \cite{LRS} results are still relevant.\par

\indent In section two, we consider tensor products of Hilbert space $H$ and a $k$-point density matrix. For $\mu_0$ a centered Gaussian measure on $H$, we express a specific integral 
$$J^{(k)} =\int_H \vert u^{\otimes k}\rangle \langle u^{\otimes k} \vert \, \mu_0(du)$$
as a series  of elementary tensors. This calculation involves combinatorial results which are expressed in terms of Knuth's odd and even decompositions of Young diagrams. 
In section three, we use concentration of measure results to show how $ 
\vert u^{\otimes k}\rangle \langle u^{\otimes k} \vert \,$ is close to its mean value $J^{(k)}$ on a  set of large probability. This statement also holds when we replace $\mu_0$ by the Gibbs measure. 
\indent In section 4 we introduce metric probability measure spaces and show how the infinite-dimensional dynamical system (\ref{canonical}) can be approximated by finite-dimensional dynamical systems, particularly involving random Fourier series. In particular, we show that $x\mapsto u(x,t)$ is $\gamma$-H\"older continuous $[0, 2\pi ]\mapsto L^4$ for $0<\gamma <1/16$. We also obtain results on the empirical distributions that arise when we sample solutions of (\ref{canonical}) with respect to Gibbs measure (\ref{Gibbs}), which we use in the numerical experiments in section 7.\par    
\indent Hasimoto observed that (\ref{canonical}) can be expressed as a Lax pair of coupled ordinary differential equations with solutions in $SO(3)$, one of which is the Serret-Frenet system for a moving frame on a curve in ${\mathbb R}^3$. In sections 5 and 6 we consider the evolution of the dynamical system corresponding to Hasimoto frames under the Gibbs measure. In section 7 we present numerical experiments regarding the solutions, which illustrate the nature of frames that arise from the solutions of (\ref{canonical}) for typical elements in the support of the Gibbs measure (\ref{Gibbs}). \par

\section{Tensor products and $k$-point density matrices for Gaussian measure}
Let $H$ be a separable complex Hilbert space, with inner product $\langle \cdot \mid\cdot \rangle$ which is linear in the second argument. We identify the injective tensor product $H\check\otimes H$ with the algebra ${\mathcal L}(H)$ of bounded linear operators on $H$ and the projective tensor product $H\hat\otimes H$ with the ideal ${\mathcal L}^1(H)$ of trace class operators on $H$. For $H=L^2$, the identification is $$f\otimes \bar g=\mid f\rangle\langle g\mid : h\mapsto f(x)\int \bar g(y)h(y)dy.$$ Let $A\in {\mathcal L}^1(H)$ be self-adjoint such that $0\leq A\leq I$, and let $\mu_0$ be a Gaussian measure on $H$ of mean zero and covariance $A$. By the spectral theorem, We can choose an orthonormal basis $(\varphi_j)_{j=1}^\infty$ of $H$ such that $A\varphi_j=\alpha_j\varphi_j$ where the spectrum of $A$ is the closure of $\{\alpha_j: j=1, 2, \dots \}$. Then we introduce mutually independent Gaussian $N(0,1)$ random variables $(\gamma_j)_{j=1}^\infty$ and the vector 
\begin{equation}\label{u}u=\sum_{j=1}^\infty \sqrt{\alpha_j}\gamma_j\varphi_j\end{equation} so that $\mu_0$ is the distribution of $u$ on $H$, as one easily checks by computing the expectation \begin{align}{\mathbb E}\exp (i\langle f, u\rangle )&={\mathbb E}\exp \Bigl( \sum_ji\sqrt{\alpha_j}\langle f, \varphi_j\rangle \gamma_j\Bigr)\nonumber\\
&=\exp\Bigl( \sum_{j} -{\frac{1}{2}}\alpha_j\langle f, \varphi_j\rangle^2\Bigr)\nonumber\\
&=\exp\bigl( -{\frac{1}{2}}\langle Af,f\rangle \bigr)\qquad (f\in H).\end{align} 
Hence $A$ is the mean of rank-one tensors with respect to Gaussian measure
$$A=\int_H \mid u\rangle \langle u\mid \, \mu_0(du).$$
\indent The $k$-fold tensor product $H^{\otimes k}$ can be completed to give a Hilbert space, so that the space $H^{s\otimes k}$ of symmetric tensors gives a closed linear subspace. We consider
$$J^{(k)}=\int_H \vert u^{\otimes k
}\rangle \langle u^{\otimes k}\vert \,\mu_0 (du).$$
An element of $H^{\otimes k}\hat \otimes H^{\otimes k}$ determines a linear operator ${\mathcal L}(H^{\otimes k})$, commonly referred to as a matrix, so $J^{(k)}\in  (L^2)^{s\otimes k}\hat \otimes (L^2)^{\otimes k}$ gives a $k$-point density matrix, or equivalently a trace class operator $J^{(k)} \in{\mathcal L}^1(H^{s\otimes k})$ .\par

\indent Lemma 3.3 of \cite{LNR1} contains calculations regarding $J^{(k)}$ which we have not been able to interpret, particularly line 8 of page 79. Here we calculate $J^{(2)}$ directly, before addressing the case of general $k$. Evidently we have ${\mathbb E}(\gamma_j\gamma_\ell\gamma_m\gamma_n)=0$ if one of the indices $j,\ell,m,n$ is distinct from all the others; otherwise, we have all the indices equal, or two distinct pairs of equal indices. Hence we have
\begin{align}
\int_H \mid u\otimes u\rangle \langle u\otimes u\mid \, \mu_0(du)&=\sum_{j, \ell, m,n} \sqrt{\alpha_j\alpha_\ell\alpha_m\alpha_n}\mid \varphi_j\otimes \varphi_\ell\rangle\langle \varphi_m\otimes\varphi_n\mid {\mathbb E} (\gamma_j\gamma_\ell\gamma_m\gamma_n)\nonumber\\
&=\sum_j 3\alpha_j^2 \mid\varphi_j\otimes \varphi_j\rangle \langle \varphi_j\otimes \varphi_j\mid\nonumber\\
&\quad+ \sum_{j,\ell: j\neq\ell} \alpha_j\alpha_\ell \mid\varphi_j\otimes \varphi_\ell\rangle \langle \varphi_j\otimes\varphi_\ell\mid\nonumber \\
&\quad+ \sum_{j,m: j\neq m}\alpha_j\alpha_m \mid \varphi_j\otimes \varphi_j\rangle\langle \varphi_m \otimes \varphi_m\mid\nonumber\\
&\quad +\sum_{j,\ell: j\neq \ell}\alpha_j\alpha_\ell\mid \varphi_j\otimes\varphi_\ell\rangle \langle \varphi_\ell\otimes\varphi_j\mid\end{align}
and we combine the second and fourth of these to obtain
\begin{align}\int_H \mid u\otimes u\rangle \langle u\otimes u\mid\, \mu_0(du)
&=\sum_j 3\alpha_j^2 \mid \varphi_j\otimes \varphi_j\rangle \langle \varphi_j\otimes \varphi_j\mid\nonumber\\
&\quad+ {\frac{1}{2}}\sum_{j,\ell: j\neq\ell} \alpha_j\alpha_\ell \mid \varphi_j\otimes\varphi_\ell+\varphi_\ell\otimes \varphi_j\rangle \langle \varphi_j\otimes\varphi_\ell +\varphi_\ell\otimes\varphi_j\mid\nonumber\\
&\quad+ \sum_{j,m: j\neq m}\alpha_j\alpha_m \mid \varphi_j\otimes \varphi_j\rangle\langle \varphi_m \otimes \varphi_m\mid\end{align}
which exhibits the right-hand side as a symmetric tensor, in which the final term shows the integral is not diagonal with respect to the orthonormal basis 
$$\Bigl\{ \varphi_j\otimes\varphi_j,\quad (\varphi_j\otimes\varphi_\ell+\varphi_\ell\otimes\varphi_j )/\sqrt{2};\quad  j, \ell\in {\mathbb N}; j\neq \ell\Bigr\}$$ 
of the symmetric tensor product $H\otimes_sH$, hence $J^{(2)}$ is not a multiple of $A\otimes A$.\par 
\indent For $k\in {\mathbb N}$, let $\Pi_k$ be the set of all partitions of $k$ so that $\pi\in \Pi_k$ may be expressed as $k=k_1+k_2+\dots +k_n$ where the row lengths $k_j\in {\mathbb N}$ have $k_1\geq k_2\geq \dots \geq k_n$. Given such $\pi$ and a $n$-element subset $\{ j_1, \dots, j_n\}$ of ${\mathbb N}$, there is a symmetric tensor 
$${\frac{1}{\sqrt{n!}}}\sum_\sigma \varphi_{\sigma (j_1)}^{\otimes k_1}\otimes \dots \otimes \varphi_{\sigma (j_n)}^{\otimes k_n}\in H^{\otimes k}$$
where the sum is over all the permutations $\sigma$ of $\{ j_1, \dots, j_n\}$. The set of all such tensors gives an orthonormal basis of the $k$-fold symmetric tensor product $H^{s\otimes k}$.\par
\indent We express $u$ as in (\ref{u}) and consider the expansion 
\begin{align}\mid u^{\otimes k}\rangle&\langle u^{\otimes k}\mid\nonumber\\
&=\sum_{(m_1, \dots ,m_{2k})\in {\mathbb N}^{2k}} \sqrt{\alpha_{m_1}\dots \alpha_{m_{2k}}}\gamma_{m_1}\dots \gamma_{m_{2k}} \mid \varphi_{m_1}\otimes \dots\otimes \varphi_{m_k}\rangle\langle \varphi_{m_{k+1}}\otimes \dots \otimes \varphi_{m_{2k}}\mid\end{align}
in terms of this orthonormal basis of $H^{s\otimes k}$, and look for the terms that do not vanish after integration with respect to $\mu_0$.  

\begin{defn} (even decomposition) Given $\pi \in \Pi_k$ consider a pair $(\lambda, \rho )\in \Pi_k^2$ with rows $\lambda: k=\ell_1+\ell_2+\dots +\ell_n$ where $\ell_j\in {\mathbb N}\cup \{ 0\}$ and $\rho: k=r_1+r_2+\dots +r_n$ where $r_j \in {\mathbb N}\cup \{0\}$ and 
$$2k_j=\ell_j+r_j\qquad (j=1, \dots , n),$$
so that $\lambda$ and $\rho$ have equal numbers of odd rows;
here rows may have zero lengths, and the rows are not necessarily in decreasing order. We refer to $(\lambda ,\rho )$ as an even decomposition of $\pi$. \end{defn}
\begin{rem} There are various alternative descriptions of even decompositions. We write $\lambda\sim \rho$ if $\lambda$ and $\rho$ are partitions that have equal numbers of boxes and equal numbers of odd rows; evidently $\sim$ is an equivalence relation on the set of partitions. By \cite{K} Theorem 4 there is a bijection between symmetric matrices $A$ that have entries in ${\mathbb N}\cup \{ 0\}$ with column sums $c_1, \dots ,c_n$ and Young tableaux $P$ such that have $c_j$ occurrences of $j$ as entries and number of columns of $P$ of odd length equals the trace of $A$.
Given symmetric matrices $A$ and $B$ with entries in ${\mathbb N}\cup \{ 0\}$ such that $A$ and $B$ have equal traces and equal totals of entries, then the $RSK$ correspondence takes $A$ to $P$ and $B$ to $Q$ where $P$ and $Q$ are Young tableaux with an equal number of boxes, and their transposed diagrams $P'$ and $Q'$ have an equal number of odd rows, so $P'\sim Q'$.\par
\end{rem}

For notational convenience, we also regard $\varphi_j^{\otimes 0}$ as a 
factor which may be omitted in tensor products. Then given such a triple $(\pi ,\lambda, \rho )$ and an $n$-subset $\{j_1, \dots ,j_n\}$ of ${\mathbb N}$,
\begin{equation}\label{term}\alpha_{j_1}^{k_1}\dots \alpha_{j_n}^{k_n}\mid \varphi_{j_1}^{\otimes \ell_1}\otimes \dots \otimes \varphi_{j_n}^{\otimes \ell_n}\rangle\langle \varphi_{j_1}^{\otimes r_1}\otimes \dots \otimes \varphi_{j_n}^{\otimes r_n}\mid {\mathbb E} \bigl( \gamma_{j_1}^{2k_1}\gamma_{j_2}^{2k_2}\dots \gamma_{j_n}^{2k_n}\bigr)\end{equation}
where
\begin{equation}\label{normalmoments}{\mathbb E} \bigl( \gamma_{j_1}^{2k_1}\gamma_{j_2}^{2k_2}\dots \gamma_{j_n}^{2k_n}\bigr)=\prod_{j=1}^n {\frac{(2k_j)!}{2^{k_j}k_j!}}\end{equation}
gives a nonzero summand in $J^{(k)} $. \par
\indent Conversely, let $(\lambda, \rho )\in \Pi_k^2$ and suppose that $\lambda$ and $\rho$ have equal numbers of odd rows, so that after adding zero rows and reordering the rows we have $r_j+\ell_j$ even for all $j$. Then we introduce $2k_j=\ell_j+r_j$ and after a further reordering write $k=k_1+k_2+\dots +k_n$ where $k_j\in {\mathbb N}$ have $k_1\geq k_2\geq \dots \geq k_n$, and we have $\pi \in \Pi_k$ as above. Given a $n$-subset $\{j_1, \dots, j_n\}$ of ${\mathbb N}$, we take $2k_m$ copies of $j_m$ and split them as $\ell_m$ on the bra side and $r_m$ on the ket side of the tensor for $m=1, \dots ,n$, making a contribution as in (\ref{term}). We summarize these results as follows.\par

\begin{prop} The integral $J^{(k)}$ is the sum over the summands (\ref{term}) that arise from a $\pi\in \Pi_k$ with $n$ nonzero rows, a $n$-subset of ${\mathbb N}$, and an even decomposition of $\pi$ into a pair $(\lambda, \rho )\in \Pi_k^2$ where $\lambda$ and $\rho$ have equal numbers of odd rows.\end{prop}

\section{Concentration of $k$-point matrices for Gibbs measure}
 Let $H=L^2({\mathbb T};{\mathbb R})$ and let $(\gamma_j)_{j\in {\mathbb Z}}$ be mutually independent Gaussian $N(0,1)$ random variables, where $z_j=(\gamma_j+i\gamma_{-j})/\sqrt{2})$ and $z_{-j}=(\gamma_j-i\gamma_{-j})/\sqrt{2}$ for $j\in {\mathbb N}$. 
Then we take Brownian loop in the style
\begin{equation}\label{loop}u(\theta )=\sum_{j\in {\mathbb Z}\setminus\{ 0\}} {\frac{z_j e^{ij\theta}}{\vert j\vert }},\end{equation}
so that Wiener loop $W(du)$ is the distribution of $u\in H$. By orthogonality, we have
\begin{equation}\label{loopnorm}\int_{\mathbb T} \vert u(\theta )\vert^2 {\frac{d\theta}{2\pi}}=\sum_{j\in{\mathbb Z}\setminus \{0\}} {\frac{\vert z_j\vert^2}{j^2}}\end{equation}
so that
$u\in B_{K/2\pi}$ if and only if $\sum_j\gamma_j^2/j^2\leq K.$ Chebyshev's inequality and independence, we have
\begin{equation}\label{fourierbound}{\mathbb P}\Bigl[ \sum_{j\in{\mathbb Z}\setminus \{0\}} {\frac{\gamma_j^2}{j^2}}\geq K\Bigr]\leq e^{-tK}{\frac{\pi \sqrt{2t}}{\sin (\pi \sqrt{2t})}}\qquad (0<t<1/2, K>0).\end{equation}
The low Fourier modes are the predominant terms since one has the estimate 
\begin{equation}\label{fouriertailbound}{\mathbb P}\Bigl[ \sum_{j\in{\mathbb Z}; \vert j\vert \geq m} {\frac{\gamma_j^2}{j^2}}\geq K\Bigr]\leq \exp\bigl( m-Km^2/4\bigr),\qquad (m\in {\mathbb N}, K>0),\end{equation}
which also follows from Chebyshev's inequality and independence.\par

Let $\mu_\lambda (du)=\zeta (\lambda )^{-1}\exp (\lambda  V(u)) W(du)$ where 
$$\zeta (\lambda)=\int_{B_K} \exp (\lambda V(u)) W(du)$$
so that $\mu_\lambda$ is a probability measure; we can take $V(u)=\int_{\mathbb T} u(\theta )^4 d\theta/(2\pi)$ and $W$ to be Brownian loop measure. Here $\mu_\lambda$ is Gibbs measure (\ref{Gibbs}) with the inverse temperature $\beta$, but we prefer to work with $\lambda=-\beta >0$ so that the convexity statements are easier to interpret. 

\begin{thm} Under the family of Gibbs measures (\ref{Gibbs}) associated with $NLS$ (\ref{NLS}), the random variable $u\mapsto \langle u^{\otimes k}\mid T\mid u ^{\otimes k}\rangle $ with $u\in (B_K, L^2, \mu_\lambda )$ and $T\in {\mathcal L}(H^{s\otimes k})$ satisfies a Gaussian concentration of measure, the mean is a Lipschitz continuous function of $\beta$, and the mean for $\beta=0$ is a sum over partitions of $2k$ over even decompositions.\end{thm}

The statements in this theorem will be proved in this section. They involve the integral
\begin{equation}\label{Gk}G_\lambda^{(k)}=\int_{B_K}\mid u^{\otimes k}\rangle\langle u^{\otimes k}\mid \mu_\lambda (du)\end{equation}
\noindent where $\mu_\lambda$ is the Gibbs measure for NLS. In the defocussing case, the $k$-particle density matrix of an interacting quantum system with suitable initial conditions converges to its classical analogue see \cite{AGT} (2.16) for the 1D case and \cite{LNR2} for 2D and 3D. \par
\indent We can write $u=P+iQ$ for real variables $(P,Q)$ and interpret $\langle u^{\otimes k}\mid T\mid u ^{\otimes k}\rangle$ as a homogeneous polynomial in $(p,q)$ of total degree $2k$.   
\indent The following result gives concentration of measure for Lipschitz functions on $(B_K, L^2, \mu_\lambda )$, and shows that $k$-point matrices are concentrated near to their mean value.\par 
\begin{prop}
For $T\in {\mathcal L}(H^{s\otimes k})$ with operator norm $\Vert T\Vert$, let $g_T:B_K\rightarrow {\mathbb C}$ by $g_{T}(u)=\langle u^{\otimes k}\mid T\mid u^{\otimes k}\rangle$. Then 
there exists $\alpha =\alpha (\beta,K)>0$ such that\end{prop}
\begin{equation}\mu_\lambda \Bigl(\bigl\{ u\in B_K: \vert g_{T}(u)-\trace(G_\lambda^{(k)}T)\vert>r\bigr\}\Bigr) \leq 4\exp \Bigl( {\frac{-\alpha r^2}{32k^2K^{2k-1}\Vert T\Vert^2}}\Bigr)\qquad (r>0).\end{equation}

\begin{proof} Here $g_{T}$ has mean value
\begin{equation}\label{mean}\int_{B_K}g_T(u) \mu_\lambda (du)=\int_{B_K}\langle u^{\otimes k}\mid T\mid u^{\otimes k}\rangle \mu_\lambda(du)=\trace (G_\lambda^{(k)}T).\end{equation}
Also $g_T$ is Lipschitz, with
\begin{align} \vert g_T(u)-g_T(v)\vert&\leq \Vert T\Vert \sum_{j=0}^{2k-1} \Vert u\Vert^{j}\Vert v\Vert^{2k-j-1}\Vert u-v\Vert\nonumber\\
&\leq 2kK^{k-1/2}\Vert T\Vert \Vert u-v\Vert\qquad (u,v\in B_K).\end{align}
By the logarithmic Sobolev inequality Theorem 1.2(iv) of \cite{BBD}, there exists $\alpha=\alpha (K,\beta )>0$ such that
\begin{equation}\label{logsob}\int_{B_K}f(u)^2\log f(u)^2\mu_\lambda (du)\leq \int_{B_K} f(u)^2\mu_\lambda (du)\log \Bigl(\int_{B_K} f(u)^2\mu_\lambda (du)\Bigr)+{\frac{2}{\alpha}}\int_{B_K} \Vert\nabla f\Vert^2 \mu_\lambda (du)\end{equation}
for all continuously differentiable $f: B_K\rightarrow {\mathbb R}$. In particular, we choose 
$$f(u)=\exp \bigl( \lambda\Re \bigl( g_T(U)-\trace(G_t^{(k)}T)\bigr) /2\Bigr)$$
and we deduce that the moment generating function
\begin{equation}\label{Cheb}\varphi (r)=\int_{B_K} \exp\Bigl(r\Re \bigl( g_T(u)-\trace(G_\lambda^{(k)}T)\bigr)\Bigr)\mu_\lambda (du)\end{equation}
satisfies $\varphi (0)=1$, $\varphi'(0)=1$ and the differential inequality
\begin{equation}\label{diffineq}r\varphi' (r)\leq \varphi (r)\log\varphi (r)+{\frac{8k^2K^{2k-1}\Vert T\Vert^2r^2}{\alpha}}\end{equation}
hence
$$\varphi (r) \leq \exp\Bigl( {\frac{8k^2K^{2k-1}\Vert T\Vert^2r^2}{\alpha}}\Bigr).$$
One can conclude the proof by a standard application of Chebyshev's inequality to the integral for $\varphi$ in (\ref{Cheb}).
\end{proof}
To make full use of the previous result, one needs to know the mean $\trace(G_\lambda^{(k)}T)$ as in (\ref{mean}), which depends upon the measure in (\ref{Gk}). The following shows how the mean can vary with the inverse temperature $\beta =-\lambda$. 
\begin{prop}
For $g:B_K\rightarrow {\mathbb{R}}$ an $L$-Lipschitz function, the mean values of $g$ with respect to the measures $\mu_\lambda$ satisfy 
\begin{equation}\Bigl( \int_{B_K}g(u)\bigl(\mu_b(du)-\mu_a(du)\bigr)
\Bigr)^2\leq {\frac{L^2(b-a)^2}{2\alpha }}\int\!\!\!\int_{B_K\times B_K}(V(u)-V(w))^2\mu_\lambda(du)\mu_\lambda(dw)\end{equation}
where $\alpha$ is the constant in (\ref{logsob}) for $\mu_a$, and some $\lambda\in (a,b)$.

\end{prop}
\begin{proof} We observe that $\log \zeta (\lambda)$ is a convex function of $\lambda>0$ and by the mean value theorem, there exists $a<\lambda<b$ such that 
\begin{align} \label{logconvex}\log\zeta (a)-\log\zeta (b)&=-(b-a){\frac{\zeta'(b)}{\zeta(b)}}+{\frac{(b-a)^2}{2}}\Bigl({\frac{\zeta''(\lambda)}{\zeta (\lambda)}}-{\frac{\zeta'(\lambda )^2}{\zeta (\lambda )^2}}\Bigr) \nonumber\\
&=-(b-a)\int_{B_K} V(u)\mu_b(du)\nonumber\\
&\quad +{\frac{(b-a)^2}{4}}\int\!\!\!\int_{B_K\times B_K} ( V(u)-V(w))^2\mu_\lambda(du)\mu_\lambda(dw).\end{align}
Let $W_1(\mu_a, \mu_b)$ be the Wasserstein transportation distance between $\mu_b$ and $\mu_a$ for the cost function $\Vert u-v\Vert_{L^2}$, as in page 34 of \cite{V1}. Then by duality we have
$$\Bigl\vert\int_{B_K} g(u)\mu_b(du)-\int_{B_K}g(u)\mu_a(du)\Bigr\vert\leq L W_1(\mu_b,\mu_a).$$
 By results of Otto and Villani discussed in \cite{V1} pages 291-2, the logarithmic Sobolev inequality of Theorem 1.2(iv) \cite{BBD} implies a transportation cost inequality 
\begin{equation}\label{transport}W_1(\mu_b,\mu_a)\leq
\Bigl( {\frac{2}{\alpha}}{\hbox{Ent}}(\mu_b\mid \mu_a)\Bigr)^{1/2}\end{equation}
in the style of Talagrand, where the relative entropy is
\begin{align}{\hbox{Ent}}(\mu_b\mid\mu_a)&=\int_{B_K} \log {\frac{d\mu_b}{d\mu_a}} \mu_b(da)\nonumber\\
&=(b-a)\int_{B_K} V(u)\mu_b(du) -(\log \zeta (b)-\zeta (a))\nonumber\\
&={\frac{(b-a)^2}{4}}
\int\!\!\!\int_{B_K\times B_K}(V(u)-V(w))^2\mu_\lambda(du)\mu_\lambda(dw),\end{align}
where the final step follows from (\ref{logconvex}). The stated result follows on combining these inequalities.
\end{proof}

\begin{prop} The integral $G_0^{(k)}$ from (\ref{Gk}) is the sum over the terms (\ref{term}) that arise from a $\pi\in \Pi_k$ with $n$ nonzero rows, a $n$-subset of ${\mathbb N}$, and an even decomposition of $\pi$ into a pair $(\lambda, \rho )\in \Pi_k^2$ where $\lambda$ and $\rho$ have equal numbers of odd rows.\end{prop}
\begin{proof} The measure $\mu_0$ is a Wiener loop measure restricted to $B_k$. For any sequence $(\varepsilon_n)_{n\in {\mathbb Z}}\in \{\pm 1\}^{\mathbb Z}$, the sequence $(\gamma_n)_{n\in {\mathbb Z}}$ with $\gamma_n$ mutually independent $N(0,1)$ Gaussian random variables has the same distribution as the sequence $(\varepsilon_n\gamma)_{n\in {\mathbb Z}}$. Also, the condition $\sum_{n\in{\mathbb Z}\setminus \{ 0\}}\gamma_n^2/n^2\leq K$ does not change under this transformation. Let $u_\varepsilon (\theta) =\sum'_j \varepsilon_j\zeta_je^{ij\theta}/\vert j\vert$. We therefore have
$$G_0^{(k)}=\int_{B_K}\int_{\{ \pm 1\}^{\mathbb Z}}\mid u_\varepsilon^{\otimes k}\rangle \langle u_\varepsilon^{\otimes k}\mid d\varepsilon \mu_0(du)$$
where $d\varepsilon$ is the Haar probability measure on the Cantor group $\{ \pm 1\}^{\mathbb Z}$. 
We can therefore compute the inner integral in this expression for $G_0^{(k)}$ by the same calculation that led to the corresponding statement for $J^{(k)}$, since we only used the even decomposition of partitions to derive (\ref{term}).\end{proof}  
\indent We have
\begin{equation}\Bigl[ \sum_{j\in {\mathbb Z}\setminus\{0\}}{\frac{\gamma_j^2}{j^2}}\leq K\Big]\subseteq\bigcap_{j\in {\mathbb Z}\setminus\{0\}}\Bigl[\gamma_j^2\leq Kj^2\Bigr]\end{equation}
where the sets are independent under the Gaussian measure $d{\mathbb P}$, so we have a substitute for (\ref{normalmoments})
\begin{equation}\int_{B_K}\gamma_{j_1}^{2k_1}\gamma_{j_2}^{2k_2}\dots \gamma_{j_n}^{2k_n}\mu_0(du)\leq {\mathbb P}(B_K)^{-1}\prod_{\ell=1}^n \int_{[\gamma^2\leq Kj_\ell^2]}\gamma^{2k_\ell}d{\mathbb P},\end{equation}
where ${\mathbb P}(B_K)$ satisfies (\ref{fourierbound}) and there is an approximate formula
\begin{equation}\int_{[\gamma^2\leq Kj_\ell^2]}\gamma^{2k_\ell}d{\mathbb P}={\frac{(2k_\ell )!}{2^{k_\ell}k_\ell !}}\exp\Bigl( -\Bigl({\frac{j_\ell^2K}{2}}\Bigr)^{k_\ell-1/2}{\frac{e^{-_\ell^2K/2}}{\Gamma (k_\ell+1/2)}}\Bigr).\end{equation}

\section{Concentration for metric measure spaces}
\indent In a similar spirit, we give a concentration result for $k$-fold stochastic integrals. This result is in the spirit of the integrability criteria of\cite{CM} which relates to a single variable.  Let $H^1$ be the Sobolev space of $v\in L^2({\mathbb T}; {\mathbb C})$ that are absolutely continuous with derivative $v'\in L^2({\mathbb T}; {\mathbb C})$. Let $h_j\in H^1$ for $j=1, \dots ,k$ be such that $\sum_{j=1}^k \int (h_j')^2 dx\leq 1$, and consider $\Phi: B_K^k\mapsto {\mathbb R}^k$ given by
\begin{equation}\Phi: (u_j)_{j=1}^k \mapsto \Bigl( \int u_j(x) h_j'(x)dx\Bigr)_{j=1}^k =\Bigl(-\int h_j(x) du_j\Bigr)_{j=1}^k.\end{equation}
The following result describes the distribution of this ${\mathbb C}^k$-valued random variable.
\begin{prop}\label{klogsob} Let $\nu_K$ be the probability measure on ${\mathbb C}^k$ that is induced from $\mu_{K, \beta}^{\otimes k}$ by $\Phi$. Then there exists $\alpha_K>0$ independent of $k$ such that
\begin{equation}\label{logsobG}\int_{{\mathbb C}^k} G(w)^2\log \Bigl( G(w)^2/\int G^2 d\nu_K\Bigr) \nu_K(dw) \leq {\frac{2}{\alpha_K}} \int_{{\mathbb C}^k}\Vert\nabla G(w)\Vert^2 \nu_K (dw)\end{equation}
for all $G\in C^1_{c}({\mathbb C}^k; {\mathbb R})$.
The distribution $\nu_K$ has mean $x_0$ and satisfies
\begin{equation} \int_{{\mathbb C}^k} e^{t^2\Vert x- x_0 \Vert^2/2} \nu_K(dx)\leq \Bigl(1-{\frac{t^2}{\alpha}}\Bigr)^{-k}\qquad (t^2<\alpha ).\end{equation}

\end{prop}
\begin{proof} We observe that for all $u=(u_j)_{j=1}^k\in B_K^k$ and $v=(v_j)_{j=1}^k\in B_K^k$ we have
\begin{equation}\Vert \Phi (u)-\Phi (v)\Vert^2_{{\mathbb R}^k}\leq \sum_{j=1}^k \int (h_j'(x))^2 dx\sum_{j=1}^k \int \vert u_j(x)-v_j(x)\vert^2 dx\end{equation}
so $\Phi: (B_K^k, \ell^2(L^2))\rightarrow ({\mathbb R}^k, \ell^2)$ is Lipschitz with constant one. Each metric probability space $(B_K, L^2, \mu_{K, \beta})$ satisfies a logarithmic Sobolev inequality
with constant $\alpha_K>0$ by \cite{BBD}, and the probability space $(B_K^k,\ell^2(L^2) ,\mu_{K,\beta}^{\otimes k})$ is a direct product of the metric probability spaces $(B_K, L^2, \mu_{K,\beta})$, hence also satisfies a logarithmic Sobolev inequality 

\begin{equation}\int_{B_K^k} G\circ \Phi (u)^2\log \Bigl( G\circ \Phi (u)^2/\int G\circ \Phi ^2 d\mu_{K, \beta}^{\otimes k}\Bigr) \mu_{K, \beta}^{\otimes k}(du) \leq {\frac{2}{\alpha_K}} \int_{B_K^k}\Vert\nabla G\circ \Phi (u)\Vert^2 \mu_{K, \beta}^{\otimes k} (du)\end{equation}  
with constant $\alpha_K$ independent of $k$, by section 22 of \cite{V2}.

We introduce $x_0=\int_{{\mathbb C}^k} x \nu_K(dx)$ and consider
\begin{equation} \varphi (t)=\int_{{\mathbb C}^k} e^{t\Re \langle x-x_0, y\rangle}\nu_K(dx)\end{equation}
which satisfies $\varphi (0)=1$, $\varphi'(0)=0$ and the differential inequality
\begin{equation} t\varphi'(t)\leq \varphi (t)\log\varphi(t)+{\frac{t^2\Vert y\Vert^2}{2\alpha}}\end{equation}
follows from (\ref{logsobG}). This gives
\begin{equation}\int_{{\mathbb C}^k} e^{t\Re \langle x-x_0,y\rangle} \nu_K(dx)\leq e^{t^2\Vert y\Vert^2/(2\alpha )}\end{equation}
which we integrate against $e^{-\Vert y\Vert^2/2}$, where $y\in {\mathbb C}^k ={\mathbb R}^{2k}$, to obtain
\begin{equation} \int_{{\mathbb C}^k} e^{t^2\Vert x-x_0\Vert^2/2} \nu_K(dx)\leq \Bigl(1-{\frac{t^2}{\alpha}}\Bigr)^{-k}\qquad (t^2<\alpha ).\end{equation}
\end{proof}

The probability space $(B_K,L^2, \mu_{K, \beta })$ has a tangent space associated with infinitesimal translations. Let $H^1$ be the Sobolev space of $v\in L^2({\mathbb T}; {\mathbb C})$ that are absolutely continuous with derivative $v'\in L^2({\mathbb T}; {\mathbb C})$; let $H^{-1}=(H^1)^*$ be the linear topological dual space for the pairing $\langle v, w\rangle \mapsto \int_{\mathbb T} v(x)\overline{w(x)} dx/(2\pi)$ as interpreted via Fourier series. Then there is a Radonifying triple of continuous linear inclusions 
\begin{equation}\label{triple}H^1\rightarrow L^2\rightarrow (H^1)^*\end{equation}
associated with the Gibbs measure $\mu_K$. The space $H^1$ has orthonormal basis\par \noindent $(e^{in\theta}/\sqrt{n^2+1})_{n=-\infty}^\infty$ and the covariance matrix of Wiener loop is $R_0={\hbox{diag}}[ 1/(1+n^2)]_{n=-\infty}^\infty$ with respect to this basis. By simple estimates, one can deduce that there exists, for each $\beta$ and $K>0$, a self-adjoint, nonnegative and trace class operator $R$ such that
\begin{equation} \bigl\langle Rf,g\bigr\rangle_{H^1} =\int_{B_K}\, \int_{[0, 2\pi ]} f'ud\theta \int_{[0, 2\pi ]} \bar g'\bar ud\theta\,  \mu_{K, \beta} (du),\end{equation}
which gives the covariance matrix of the Gibbs measure on $H^1$. This is essentially $G^{(1)}_{-\beta}$, up to the identification of Hilbert spaces in (\ref{triple}).\par
\indent Cameron and Martin computed the density with respect to the Wiener measure that results from the linear translation $u\mapsto u+v$ for $v\in H^1$; their results extends to Gibbs measure with some modifications.

We momentarily suppress the dependence of functions upon time, and consider for $p,q\in H^1$, the linear transformation $P+iQ\mapsto P+p+i(Q+q)$. 
Cameron and Martin proved that free Wiener measure $(\beta =0)$ is mapped to a measure that absolutely continuous with respect to the free Wiener measure. Likewise, Gibbs measure is mapped to a measure absolutely continuous with respect to Gibbs measure, so we can regard $H^1$ as fibres of a tangent space to $(B_K, L^2, \mu_{K, \beta })$.  
We make a polar decomposition $P+iQ=\kappa e^{i\sigma}$ with $\kappa=\sqrt{P^2+Q^2}$ and consider $\tau={\frac{\partial \sigma}{\partial x}}$. 
\begin{prop} For $p,q\in H^1$  the
functional 
\begin{equation} L(P,Q)= L( \kappa e^{i\sigma })=-\int_{\mathbb T} {\frac{\partial p}{\partial x}}\kappa \cos\sigma dx -\int_{\mathbb T} {\frac{\partial q}{\partial x}}\kappa\sin\sigma\, dx\end{equation}
is a Lipschitz functional of $P+iQ=\kappa e^{i\sigma}$ such that 
\begin{align}
\int_{B_K} \exp \Bigl( s L(P,Q)-&s\int_{B_K}  L(P,Q)\mu_{K, \beta } (dPdQ)\Bigr) \mu_{K, \beta}(dPdQ)\nonumber\\
&\leq \exp \Bigl\{ Cs^2\int_{\mathbb T} 
\Bigl(\Bigl( {\frac{\partial p}{\partial x}}\Bigr)^2+ \Bigl( {\frac{\partial q}{\partial x}}\Bigr)^2\Bigr)dx\Bigr\}\qquad (s\in {\mathbb R}). \end{align}
\end{prop}\vskip.05in
\noindent \begin{proof} Note that $P+iQ\mapsto \kappa$ is $1$-Lipschitz with
$$\bigl[ \sigma\kappa , \kappa\nabla\sigma\bigr]={\frac{1}{\sqrt{P^2+Q^2}}}\begin{bmatrix} P&-Q\\ Q&P\end{bmatrix}\in SO(2).$$
Also $\kappa^2\Vert {\hbox{Hess}}\, \sigma\Vert$ is bounded. We have
\begin{equation} L(P,Q)
=\int p\Bigl( {\frac{\partial \kappa }{\partial x}}\cos\sigma -\kappa \tau\sin\sigma \Bigr)dx +\int q\Bigl( {\frac{\partial \kappa }{\partial x}}\sin\sigma +\kappa \tau\cos\sigma \Bigr)dx\end{equation}
which is bounded on $L^2$ with norm $\Lambda$ where 
\begin{equation}\Lambda^2\leq \int_{\mathbb T} \Bigl(\Bigl( {\frac{\partial p}{\partial x}}\Bigr)^2+ \Bigl( {\frac{\partial q}{\partial x}}\Bigr)^2\Bigr)dx.\end{equation}
By the concentration of measure theorem for $\nu_K$, we deduce the stated inequality. \end{proof}

\noindent {\bf Definition} We say that $(M,d,\mu )$ satisfies $T_2(\alpha )$ if 
\begin{equation}\label{T2inequality} W_2(\nu, \mu )^2\leq {\frac{2}{\alpha }}{\hbox{Ent}}(\nu \mid \mu )\end{equation}
for all probability measures $\nu$ that are of finite relative entropy with respect to $\mu $. The notation credits Talagrand, who developed the theory of such transportation inequalities. Otto and Villani showed that $LSI(\alpha )$ implies $T_2(\alpha )$; see \cite{V1} and \cite{V2}\par
\vskip.05in

\begin{thm}\label{sanov} Let $(M,d,\mu )$ be a metric probability space that satisfies $T_2(\alpha )$; let\par
\noindent $(M^N, \ell^2(d), \mu^{\otimes N})$ be the direct product metric probability space. Let $L_N^\xi=N^{-1}\sum_{j=1}^N \delta_{\xi_j}$ be the empirical distribution for $\xi=(\xi_j)_{j=1}^N\in M^N$ where $\xi_j$ distributed as $\mu$. Then the concentration inequality holds
\begin{equation}\label{eq:sanov}\mu ^{\otimes N}\Bigl(\Bigl\{ \xi: \in M^N: \bigl\vert W_p(L^\xi_N, \mu )-{\mathbb E}W_p(L_N,\mu )\bigr\vert >\varepsilon \,\Bigr\}\Bigr)\leq 2e^{-N\alpha \varepsilon^2/2}\qquad (\varepsilon >0)\end{equation}
for $p=1,2$.\end{thm} 
\begin{proof}
The map between metric spaces
\begin{equation}L_N : (M^N, \ell^2(d))\rightarrow ({\hbox{Prob}}(M), W_2):\quad (x_j)_{j=1}^N\mapsto {\frac{1}{N}}\sum_{j=1}^N \delta_{x_j}\end{equation}
associated with the empirical distribution is $1/\sqrt{N}$-Lipschitz.  Let $(x_j)_{j=1}^N, (y_j)_{j=1}^N\in M^N$ and consider the probability measure on $M\times M$ given by
\begin{equation}\label{plan}\pi ={\frac{1}{N}}\sum_{j=1}^N\delta_{(x_j, y_j)}\end{equation}
which has marginals $L_N^x=N^{-1}\sum_{j=1}^N\delta_{x_j}$ and 
$L_N^y=N^{-1}\sum_{j=1}^N\delta_{y_j}$, hence gives a transport plan with cost 
\begin{equation}\label{Lipschitz}W_2(L_N^x,L_N^y)^2\leq \int\!\!\int_{M\times M}d(x,y)^2\pi (dxdy)={\frac{1}{N}}\sum_{j=1}^Nd(x_j,y_j)^2.\end{equation}
\indent 
Suppose that $(M,d,\mu )$ satisfies $T_2(\alpha )$. Then we take $N$ independent samples $\xi_1, \dots , \xi_N$, each distributed as $\mu$ so they have joint distribution $\mu^{\otimes N}$ on $(M^N, \ell^2(d))$, where by independence \cite{BB} Theorem 1.2, $(M^N, \ell^2(d), \mu^{\otimes N})$ also satisfies $T_2(\alpha )$. By forming the empirical distribution, we obtain a map $L_N:(M^N, \ell^2(d))\rightarrow ({\hbox{Prob}}M, W_2)$. Then  $\varphi (\xi) =\sqrt{N}W_p(L^\xi_N,\mu )$ is $1$-Lipschitz $(M^N, \ell^2(d))\rightarrow {\mathbb R}$, since by the triangle inequality and (\ref{Lipschitz}),
\begin{align}\vert \varphi (\xi) -\varphi (\eta) \vert&\leq \sqrt{N}W_p(L_N^\xi, L_N^\eta )\nonumber\\
&\leq \sqrt{N}W_2(L_N^\xi, L_N^\eta )\nonumber\\
&\leq \sum_{j=1}^N d(\xi_j,\eta_j)^2;\end{align}
hence $\varphi$ satisfies the concentration inequality
\begin{equation} \int_{M^N} \exp \Bigl( t\varphi (\xi ) -t\int \varphi d\mu^{\otimes N}\Bigr)d\mu^{\otimes N} \leq \exp (t^2/(2\alpha ))\qquad (t\in{\mathbb R}).\end{equation}
\noindent Then the stated concentration inequality follows from Chebyshev's inequality.\par
\end{proof}
\indent Theorem \ref{sanov} gives a metric version of Sanov's theorem on the empirical distribution; see page 70 of \cite{DS}. There are related results in Bolley's thesis \cite{Bl}. By \cite{CEMS}, Theorem \ref{sanov} applies to Haar probability measure on $SO(3)$ and normalized area measure on ${\mathbb S}^2$, as is relevant in section 7 below. However, to ensure that
${\mathbb E}W_2(L_N, \mu )\rightarrow 0$ as $N\rightarrow \infty$, it is convenient to reduce to one-dimensional distributions, where we use the following integral formula. For distributions $\mu$ and $\nu$ on ${\mathbb R}$ with cumulative distribution functions $F$ and $G$, we write $W_p(\mu, \nu )=W_p(F,G)$.\par

\begin{prop}\label{prop:fluctuation} Let $\xi_1$ be a real random variable with finite fourth moment, and cumulative distribution function $F$. Let $\xi_1, \dots ,\xi_N$ be mutually independent copies of $\xi_1$ giving an empirical measure $L^\xi_N=N^{-1}\sum_{j=1}^N\delta_{\xi_j}$ with cumulative distribution function $F^\xi_N(t).$\par
\indent Then
\begin{equation} \int_{{\mathbb R}^N}W_1(F^\xi_N,F)\mu^{\otimes N}(d\xi )\leq {\frac{1}{\sqrt{N}}}\int_{-\infty}^\infty \sqrt{F(t)(1-F(t))}dt.\end{equation}
\end{prop}

\begin{proof} Let $H$ be Heaviside's unit step function; then $F^\xi_N(t)=N^{-1}\sum_{j=1}^N H(t-\xi_j)$, so $\sqrt{N}(F_N(t)-F(t))$ is a sum of mutually independent and bounded random variables with mean zero. Also, as in the weak law of large numbers, we have
\begin{align} \label{WLLN}{\mathbb E}W_1(F_N,F)&=\int_{-\infty}^\infty {\mathbb E}\vert F_N(t)-F(t)\vert\,dt\nonumber\\
&\leq\int_{-\infty}^\infty \Bigl( {\mathbb E}\bigl( (F_N(t)-F(t))^2\bigr)\Bigr)^{1/2}\, dt\nonumber\\ 
&={\frac{1}{\sqrt{N}}}\int_{-\infty}^\infty \sqrt{F(t)(1-F(t))}\, dt,\end{align}
where the integral is finite by Chebyshev's inequality since ${\mathbb E}\xi^4$ is finite. Compare \cite{BGV}.\par
\end{proof}

\begin{prop}\label{prop:fluctuationS2}
Suppose that $\xi_1$ has distribution $\mu$ on ${\mathbb S}^2$ where $\mu$ is absolutely continuous with respect to the normalized area measure $\nu_1$, and $d\mu=fd\nu_1$ where $f$ is bounded with $\Vert f\Vert_\infty\leq M$. Let $\xi_j$ be mutually independent copies of $\xi_1$, and let $L_N^\xi$ be the empirical measure from $N$ samples. Then 
$$\int_{({\mathbb S}^2)^N}W_1(L_N^\xi , \mu )d\mu^{\otimes N}=O(N^{-1/4})\qquad (N\rightarrow\infty ).$$\end{prop}
\begin{proof}
Let $g:{\mathbb S}^2\rightarrow {\mathbb R}$ be $1$-Lipschitz,
and suppose without loss of generality that $g$ has $\int_{{\mathbb S}^2} g(x)\nu_1(dx)=0$; then $g$ is bounded with $\Vert g\Vert_\infty \leq \pi$. Given $\delta>0$, by considering
 squares for coordinates in longitude and colatitude, we choose 
disjoint and connected subsets $E_\ell$ with diameter ${\hbox{diam}}(E_\ell )\leq \delta$ and $\nu_1(E_\ell )\leq \delta^2$ and $\mu (E_\ell)\leq M\nu_1(E)
$ such that $\cup_\ell E_\ell ={\mathbb S}^2$. We can arrange that there are $S_\delta$ such sets $E_\ell$, where $S_\delta\leq C/\delta^2$. Let ${\mathcal F}$ be the $\sigma$-algebra that is generated by the $E_\ell$,
take conditional expectations in $L^2(\nu_1)$,
and observe that
\begin{align} \int_{{\mathbb S}^2} g(x)dL_N^\xi -\int_{{\mathbb S}^2} g(x) d\mu (x)&=\int_{{\mathbb S}^2} \bigl(g(x)-{\mathbb{E}}g\mid {\mathcal F})
\bigr) dL_N^\xi + \int_{{\mathbb S}^2} \bigl({\mathbb{E}}(
g\mid {\mathcal F}
)-g(x)\bigr) d\mu (x)
\nonumber\\
&\quad +
\int_{{\mathbb S}^2} {\mathbb{E}}
(g\mid {\mathcal F})(x)\bigl( 
dL_N^\xi(x)-d\mu (x)\bigr)
\end{align}
where we have bounds
\begin{align}\label{local}\Bigl\vert 
\int_{{\mathbb S}^2} \bigl(g(x)-{\mathbb{E
}}(g\mid {\mathcal F})
\bigr) dL_N^\xi\Bigr\vert &\leq \sup_\ell
{\hbox{Lip}}(g){\hbox{diam}}(E_\ell
)\leq \delta,\\
\Bigl\vert 
\int_{{\mathbb S}^2} \bigl(g(x)-{\mathbb{E}}(g\mid {\mathcal F})
\bigr) d\mu\vert &\leq \sup_\ell
{\hbox{Lip}}(g){\hbox{diam}}(E_\ell 
)\leq \delta,\end{align}
and the identity
$$
\int_{{\mathbb S}^2} {\mathbb{E}}(g\mid {\mathcal F})(x)\bigl( 
dL_N^\xi(x)-d\mu (x)\bigr)=\sum_\ell
{\frac{\int_{E_\ell}g(x)d\nu_1(x)}{\nu_1(E_\ell
)}}\int_{E_\ell}(dL_N^\xi -d\mu ),$$
so by Cauchy--Schwarz, we have
\begin{align}\int_{({\mathbb S}^2)^N}&\Bigl( \int_{{\mathbb S}^2} {\mathbb{E}}(g\mid {\mathcal F})(x)\bigl( 
dL_N^\xi(x)-d\mu (x)\bigr)\Bigr)^2d\mu^{\otimes N}\nonumber\\
&\leq \sum_\ell\nu_1(E_\ell )\Vert g\Vert_\infty^2 \sum_{\ell}
{\frac{1}{\nu_1(E_\ell
)}} \int_{({\mathbb S}^2)^N} (L_N^\xi (E_\ell )-\mu (E_\ell ))^2
d\mu^{\otimes N}\end{align}
where $L_N^\xi (E_\ell )-\mu (E_\ell )=N^{-1}
\sum_{j=1}^N( {\mathbb I}_{E_\ell}(\xi_j)-\mu (E_\ell ))$ is a sum of independent random variables with mean zero and variance 
$ N^{-1}\mu (E_\ell )(1-\mu(E_\ell ))$, so
\begin{align}\label{average}\int_{({\mathbb S}^2)^N} \Bigl( \int_{{\mathbb S}^2} {\mathbb E}(g\mid {\mathcal F})(x)\bigl( 
dL_N^\xi(x)-d\mu (x)\bigr)\Bigr)^2d\mu^{\otimes N}\leq 
{\frac{1}{N}}\Vert g\Vert_\infty^2\sum_\ell M\leq {\frac{\pi^2CM
}{\delta^2N}}.\end{align}
Choosing $\delta =N^{-1/4}$ we make both (\ref{local}) and (\ref{average}) small, which gives the stated result.\par
\end{proof}
\noindent{\bf Remark} Consider the discrete metric $\delta $ on $[0,1]$, and observe that ${\mathbb I}_A$ gives a $1$-Lipschitz function on $[0,1]$ for all open $A\subseteq [0,1]$. Then we have $\int{\mathbb I}_A(x)(d\mu (x)-d\nu (x))=\mu (A)-\nu (A)$, so by maximizing over $A$ we obtain the total variation norm $\Vert \mu -\nu\Vert_{var}$. With $\mu$ a continuous measure and $\nu$ a purely discrete measure, such as an empirical measure, we have $\Vert \mu-\nu\Vert_{var}=1$. The Propositions \ref{prop:fluctuation} and \ref{prop:fluctuationS2} depend upon the choice of cost function as well as the measures.\par

\indent The Gibbs measure (\ref{Gibbs}) was defined using random Fourier series. This construction gives us a sequence of finite-dimensional probability spaces which approximate the space $(B_K,L^2,\mu_{K,\beta })$. To make this idea precise, we recall some definitions from \cite{S}. \par

\begin{defn} (Convergence of metric measure spaces) 
\indent (i) For $M$ a nonempty set, a pseudometric is a function $\delta :M\rightarrow [0, \infty ]$ such that 
\begin{equation} \delta (x,y)=\delta (y,x), \quad \delta (x,x)=0, \quad \delta (x,z)\leq \delta (x,y)+\delta (y,z)\qquad (x,y,z\in M );\end{equation}
then $(M , \delta )$ is a pseudometric space.\par
\indent (ii) Given pseudo metric spaces $(M_1, \delta_1)$ and $(M, \delta_2)$, a coupling is a pseudo metric $\delta :M\rightarrow [0, \infty]$ where $M=M_1\sqcup M_2$ such that $\delta\mid M_1\times M_1=\delta_1$ and $\delta\mid M_2\times M_2=\delta_2$.\par
\indent (iii) Suppose that $\hat M_1=(M_1, \delta_1, \mu_1)$ and $\hat M_2=(M_2, \delta_2, \mu_2)$ are complete separable metric spaces endowed with probability measures. Consider a coupling $(M , \delta )$ and a probability measure $\pi$ on $M_1\times M_2$ with marginals $\pi_1=\mu_1$ and $\pi_2=\mu_2$. Then the $L^2$ distance between $\hat M_1$ and $\hat M_2$ is
\begin{equation} {\mathcal D}_{L^2}(\hat M_1, \hat M_2)=\inf_{\delta, \pi}\Bigl( \int_{M\times M} \delta (x,y)^2\pi (dxdy)\Bigr)^{1/2}\end{equation}
\end{defn}

\indent Let $D_n$ be the Dirichlet projection taking $\sum_{k=-\infty}^\infty  (a_k+ib_k)e^{ikx}$ to  $\sum_{k=-n}^n  (a_k+ib_k)e^{ikx}$.  Following \cite{Bo2}, we truncate the random Fourier series of $u=P+iQ=\sum_{k=-\infty}^\infty (a_k+ib_k)e^{ikx}$ to $u_n=P_n+iQ_n=\sum_{k=-n}^n (a_k+ib_k) e^{ikx}$ and correspondingly modify the Hamiltonian to
\begin{equation}H_3^{(n)}((a_k),(b_k))={\frac{1}{2}}\sum_{k=-n}^n k^2(a_k^2+b_k^2)+{\frac{\beta}{4}}\int \Bigl\vert \sum_{k=-n}^n (a_k+ib_k)e^{ik\theta}\Bigr\vert^4 {\frac{d\theta}{2\pi}}\end{equation}
for the real canonical variables $((a_k,b_k))_{k=-n}^n$. Then the canonical equations become a coupled system of ordinary differential differential equations in the Fourier coefficients. We introduce the polar decomposition $P_n+iQ_n=\kappa_n e^{i\sigma_n}$, and observe that in terms of these noncanonical variables, the Hamiltonians $H_1^{(n)}=\int_{\mathbb T} \kappa_n^2d\theta$ and 
\begin{equation}H_3^{(n)}={\frac{1}{2}}\int_{\mathbb T}\Bigl( \Bigl( {\frac{\partial \kappa_n}{\partial\theta}}\Bigr)^2 +\kappa_n^2\Bigl({\frac{\partial\sigma_n}{\partial\theta}}\Bigr)^2\Bigr) {\frac{d\theta}{2\pi}}+{\frac{\beta}{4}}\int_{\mathbb T} \kappa_n^4 {\frac{d\theta}{2\pi}}\end{equation}
are invariants under the flow. The corresponding Gibbs measure is 
\begin{equation}d\mu_{K, \beta}^{(n)} =Z(K, \beta ,n)^{-1} {\bf I}_{B_K}(u) \exp\Bigl( {\frac{-\beta}{4}}\int_{\mathbb T} \vert u(\theta )\vert^4 {\frac{d\theta}{2\pi}}\Bigr) \prod_{\theta\in [0, 2\pi ]}du(\theta )\end{equation}
in which 
\begin{equation} \prod_{\theta\in [0, 2\pi ]}du(\theta )=\prod_{j=-n; j\neq 0}^n \exp \Bigl(-{\frac{j^2}{2}}(a_j^2+b_j^2)\Bigr) {\frac{j^2 da_jdb_j}{2\pi}}.\end{equation}
Consider the map $u(x,t)\mapsto u(x+h,t)$ of translation in the space variable. This commutes with $D_n$, and the Gibbs measures $\mu_{K, \beta}^{(n)}$ are all invariant under this translation. 
In terms of Fourier components, we have $M_\infty=B_K$ and 
\begin{equation}\label{Mn}M_n=\Bigl\{ (a_j,b_j)_{j=-n}^n: a_j, b_j\in {\mathbb R}: \sum_{j=-n}^n (a_j^2+b_j^2)\leq K\Bigr\}\end{equation}
with the canonical inclusions of metric spaces $(M_1, \ell^2) \subset (M_2, \ell^2)\subset \dots \subset (M_\infty, \ell^2)$ defined by adding zeros at the start and end of the sequences, which gives a sequence of isometric embeddings for the $\ell^2$ metric on sequences. When we identify $(a_j,b_j)_{j=-n}^n$ with $\sum_{j=-n}^n (a_j+ib_j)e^{ijx}$, then we have a corresponding embedding for the $L^2$ metric.\par
\indent Here $(M_n, L^2,\mu_{K, \beta}^{(n)} )$ is a finite-dimensional manifold and a metric probability space. We now show that these spaces converge to $(M_\infty ,L^2,\mu_{K, \beta} )$ as $n\rightarrow\infty$.\par

\begin{lem}\label{Mmetricconvergence} (i) Suppose that $0<-\beta K<3/(14\pi^2)$. Then $\hat M_n =(M_n, {L^2}, 
\mu_{K, \beta}^{(n)})$ has
\begin{equation} {\mathcal D}_{L^2}(\hat M_n, \hat M_\infty )\rightarrow 0\qquad (n\rightarrow\infty ).\end{equation}
(ii) The measures $\mu_{K, \beta}^{(n)}$ converge in total variation norm to $\mu_{K, \beta}$ as $n\rightarrow\infty$.
\end{lem}
\begin{proof} (i) 
This is proved in Theorem 3.2 of \cite{B2}; see also Example 3.8 of \cite{S}. 
Let $W_2(D_n\sharp \mu , \mu)$ be the Wasserstein transportation distance between $D_n\sharp\mu$ and free Brownian loop measure $\mu$ for the cost function $\Vert u-v\Vert^2_{L^2}$, where $D_n$ is the Dirichlet projection. The key point is 
\begin{align} W_2(D_n\sharp \mu , \mu )^2&\leq \int \Vert D_nu-u\Vert^2_{L^2} \mu (du)\nonumber\\
&={\mathbb E}\sum_{k: \vert k\vert>n}{\frac{\vert\gamma_k\vert^2}{k^2}}=O\Bigl({\frac{1}{n}}\Bigr)\qquad (n\rightarrow\infty ).\end{align}
\indent (ii)
The measures $\mu_{K, \beta}^{(n)}$ converge in total variation norm to $\mu_K$, by an observation of McKean in step 7 of \cite{M}. By M. Riesz's theorem, there exists $c_4>0$ such that $\int_{\mathbb T} \vert D_nu\vert^4d\theta\leq c_4\int_{\mathbb T}\vert u\vert^4d\theta$, and by \cite{LRS} the integral
\begin{equation} \int_{B_K}\exp \Bigl( \lambda c_4\int_{\mathbb T} \vert u(\theta )\vert^4 d\theta \Bigr) W(du)\end{equation}
is finite, so we can use the integrand as a dominating function to show
\begin{equation} \int_{B_K}\Bigl\vert \exp\Bigl( \lambda \int_{\mathbb T} \vert D_nu(\theta )\vert^4 d\theta\Bigr) - \exp\Bigl( \lambda \int_{\mathbb T} \vert u(\theta )\vert^4 d\theta\Bigr)\Bigr\vert W(du)\rightarrow 0\qquad (n\rightarrow\infty ).\end{equation}

\end{proof}

\begin{prop}\label{L2convergence} Let $(M_n\sqcup M_\infty, \delta_n )$ be a coupling of $(M_n, L^2)$ and $M_\infty ,L^2)$, and let $\varphi :(M_n\sqcup M_\infty, \delta_n )\rightarrow {\mathbb R}$ be a Lipschitz function. Then
\begin{equation}\label{weakconvergence} \int_{M_n} \varphi (u_n)\mu_{K, \beta}^{(n)}(du_n)\rightarrow\int_{M_\infty} \varphi (u) \mu_{K, \beta}(du)\qquad (n\rightarrow\infty ).\end{equation}\end{prop}
\begin{proof}
We can introduce a pseudo metric $\delta_n$ on $M_n\cup M_\infty$ that restricts to the $L^2$ metric on $M_n$ and $M_\infty$, and apply (\ref{metricconvergence}) to Lipschitz functions $\varphi:(M_n\sqcup M_\infty , \delta_n)\rightarrow {\mathbb R}$. 
\indent We can regard $M_n\times M_\infty$ as a subset of $M\times M=(M_n\sqcup M_\infty )\times (M_n\sqcup M_\infty)$. Note that for a Lipschitz function $\varphi :M\rightarrow {\mathbb R}$ such that $\vert \varphi (x)-\varphi (y)\vert\leq \delta (x,y)$ for all $x,y\in M$, we have
\begin{align}\label{metricconvergence}\int_{M_n} \varphi (u_n) \mu_{K, \beta}^{(n)}(du_n)-\int_{M_\infty }\varphi (u)\mu_{K, \beta}(du)&
=\int\!\!\!\int_{M_n\times M_\infty } (\varphi (u_n)-\varphi (u))\pi (dudu_n)\nonumber\\
&\leq \int\!\!\!\int_{M_n\times M_\infty} \delta (u_n,u)\pi (du_ndu)\nonumber\\
&\leq \Bigl(\int\!\!\!\int_{M_n\times M_\infty} \delta (u_n,u)^2\pi(du_ndu)\Bigr)^{1/2}\nonumber\\
&={\mathcal D}_{L^2} (\hat M_n, \hat M_\infty ).\end{align}
\end{proof}

For example, with $u=\sum_{n=-\infty}^\infty (a_k+ib_k)e^{ik\theta }$ we introduce $D_nu=\sum_{k=-n}^n (a_k+ib_k)e^{ik\theta}$; then $\varphi (u)=\Vert D_nu\Vert_{L^2}$ and $\psi (u)=\Vert u-D_nu\Vert_{L^2}$ give Lipschitz functions  $\varphi, \psi :(B_K, L^2)\rightarrow {\mathbb R}$.\par

\begin{prop} For $0<\gamma <1/16$ and fixed $0<t<t_0$, the map $x\mapsto u(x,t)\in L^4$ is $\gamma$- H\"older continuous.\end{prop}
\begin{proof} We prove that for $0<t<t_0$, we have $C=C(t_0)$ such that
\begin{equation}\label{CK}\int_{B_K} \bigl\Vert u(x+h,t)-u(x,t)\bigr\Vert^{16}_{L^4_x}\mu_{K,\beta} (du)\leq Ch^2\end{equation}
so $x\mapsto u(x,t)\in L^4_x$ is $\gamma$- H\"older continuous for $0<\gamma <1/16$ by the Censov--Kolmogorov's theorem. To obtain (\ref{CK}), let $J_{3/8}(x)=\sum' e^{ikx}/\vert k\vert^{3/8}$ so that $J_{3/8}(x)\vert x\vert^{5/8}$ is bounded on $(-\pi ,\pi )$ and $J_{3/8}\in L^{4/3}(-\pi ,\pi)$. Then by Young's inequality for convolutions, with $\vert D\vert : e^{inx}\mapsto \vert n\vert e^{inx}$ we have 
\begin{equation}\bigl\Vert u(x+h,t)-u(x,t)\bigr\Vert_{L^4_x}\leq \bigl\Vert J_{3/8}\Vert_{L^{4/3}}\bigl\Vert \vert D\vert^{3/8} u(x+h,t)-\vert D\vert^{3/8}u(x,t)\bigr\Vert_{L^2_x}\end{equation}
Then by Bourgain's estimate on the solutions of NLS \cite{B}, there exists $C(t_0)$
such that
\begin{equation} \bigl\Vert u(x+h,t)-u(x,t)\bigr\Vert_{H^{3/8}_x}\leq C(t_0)\bigl\Vert u(x+h,t)-u(x,t)\bigr\Vert_{H^{3/8}_x}\end{equation}
\noindent where
\begin{align} \int_{B_K}& \bigl\Vert u(x+h,t)-u(x,t)\bigr\Vert^{16}_{H^{3/8}_x}\mu_{K,  \beta} (du)\nonumber\\
&\leq \Bigl( \int_{B_K} \bigl\Vert u(x+h,t)-u(x,t)\bigr\Vert^{32}_{H^{3/8}_x}W(du)\Bigr)^{1/2}\Bigl( \int_{B_K} \Bigl( {\frac{d\mu_{K, \beta}}{dW}}\Bigr)^2 W(du)\Bigr)^{1/2}.\end{align}
By basic results about Gaussian series, the first factor on the right-hand side is bounded by the eighth power of 
\begin{equation} \sum_{k=1}^\infty {\frac{k^{3/4}(1-\cos hk)}{k^2}}\leq Ch^{1/4},\end{equation}
so we obtain (\ref{CK}). Also, by rotation invariance of the Gibbs measure, we have 
$$\int_{B_K} \bigl\vert u(\theta +h,t)-u(\theta ,t)\bigr\vert^{4}\mu_{K,\beta} (du)=
\int_{B_K} \bigl\Vert u(x+h,t)-u(x,t)\bigr\Vert^{4}_{L^4_x}\mu_{K,\beta} (du),$$
which is 
\begin{align*}{}&\leq  \Bigl( \int_{B_K} \bigl\Vert u(x+h,0)-u(x,0)\bigr\Vert_{L^4_x}^{8}W(du)\Bigr)^{1/2}\Bigl( \int_{B_K} \Bigl( {\frac{d\mu}{dW}}\Bigr)^2 W(du)\Bigr)^{1/2}\\
&\leq C\Bigl( \sum_{k=1}^\infty {\frac{1-\cos hk}{k^2}}\Bigr)^2\leq Ch^2,\end{align*}
so $x\mapsto u(x,t)$ is $1/4$-H\"older continuous along solutions in the support of the Gibbs measure. 
\end{proof}

\section{Hasimoto transform}

We recall the Hasimoto \cite{H} transform, which associates with a solution $u\in C^2$ of (\ref{NLS})
a space curve in ${\mathbb R}^3$ with moving frame $\{ T,N,B\}$;
Hasimoto considered the case $\beta=-1/2$. In the present context, $u$ is associated with the space derivative of a tangent vector $T$ to a unit speed space curve, so the curvature is $\kappa =\Vert {\frac{\partial T}{\partial x}}\Vert$.
We have a polar decomposition $u=\kappa e^{i\sigma}$ where $\sigma (x,t)=\int_0^x \tau (y,t)dy$ and $\tau$ is the torsion. Then the Serret--Frenet formula is
\begin{equation}\label{SerretFrenet}{\frac{\partial}{\partial x}} \begin{bmatrix} T\\ N\\ B\end{bmatrix} =\begin{bmatrix} 0&\kappa &0\\ -\kappa&0&\tau\\ 0&-\tau &0\end{bmatrix}\begin{bmatrix} T\\ N\\ B\end{bmatrix},\end{equation}
\noindent so the frame develops along the space curve. Let $X=[T;N;B]\in SO(3)$, and $\Omega_1(x,t)$ the matrix in (\ref{SerretFrenet}). The space $${\mathcal L}(SO(3))=\{ g: [0, 2\pi ]\rightarrow SO(3); g\quad {\hbox{continuous}},\quad g(0)=g(2\pi)\}$$ with pointwise multiplication is a loop group, and its Lie algebra may be regarded as 
\begin{align*}H_0^1(so(3))=\Bigl\{ h;[0, 2\pi ]\rightarrow  so(3);& h\quad {\hbox{absolutely continuous}},\quad h(0)=h(2\pi )=0,\\
&\quad \int_0^{2\pi} \Vert h'(x)\Vert^2_{so(3)} dx<\infty\Bigr\};\end{align*}
see \cite{DL}. When $\Omega_1 (\cdot , t) \in C({\mathbb T}; so(3))$, the solution $X(\cdot , t)\in C([0, 2\pi ]; SO(3))$ to ({\ref{SerretFrenet}}) is $2\pi$ periodic up to a multiplicative monodromy factor $U(t)\in SO(3)$ such that $X(x+2\pi ,t)=X(x,t)U(t)$.\par
\indent The frame also evolves with respect to time, so that with $\mu =-{\frac{\partial\sigma}{\partial t}}-\beta\kappa^2$, we have 
\begin{equation}\label{Lax}{\frac{\partial}{\partial t}} \begin{bmatrix} T\\ N\\ B\end{bmatrix} =\begin{bmatrix} 0&-\tau\kappa &{\frac{\partial \kappa}{\partial x}}\\ \tau\kappa&0&-\mu\\ -{\frac{\partial \kappa}{\partial x}}&\mu &0\end{bmatrix}\begin{bmatrix} T\\ N\\ B\end{bmatrix}.\end{equation}

\begin{lem} (Hasimoto) If $u$ is a $C^2$ function that satisfies the nonlinear Schr\"odinger equation, then the coupled pair of differential equations is consistent, giving a Lax pair.\end{lem}

Thus the frame $X\in SO(3)$ evolves along the solution $P+iQ\in B_K$ of NLS, and we can regard $d/dx-\Omega_1$ and $d/dt-\Omega_2$ as connections for this evolution. Both of the coefficient matrices are real and skew symmetric. One can check that a solution of the integral equation
\begin{equation}X(x,t)=X_0(x)+t\Omega_2(0,0)X_0(0)+\int_0^x\int_0^t \Bigl( {\frac{\partial\Omega_1 (y,s)}{\partial t}}+\Omega_1(y,s)\Omega_2(y,s)\Bigr)X(y,s)dsdy\end{equation}
satisfies
$$X(x,0)=X_0(x),\qquad  {\frac{\partial X(x,0)}{\partial t}}=\Omega_2(x,0)X_0(x),$$
$${\frac{\partial^2X(x,t)}{\partial x\partial t}}= \Bigl( {\frac{\partial\Omega_1 (x,t)}{\partial t}}+\Omega_1(x,t)\Omega_2(x,t)\Bigr)X(x,t),$$
so smooth solutions are given in terms of an integral equation.\par
\indent 
From the Serret--Frenet formulas the components of the acceleration along the space curve satisfy
\begin{align}\label{acceleration}\Bigl\Vert T\times {\frac{\partial^2T}{\partial x^2}}\Bigr\Vert^2&=\Bigl({\frac{\partial \kappa}{\partial x}}\Bigr)^2 +\kappa^2\tau^2=  \Bigl( {\frac{\partial Q}{\partial x}}\Bigr)^2+\Bigl( {\frac{\partial P}{\partial x}}\Bigr)^2,\nonumber\\
\Bigl( T\cdot {\frac{\partial^2T}{\partial x^2}}\Bigr)^2&=\kappa^4= \bigl( P^2+Q^2\bigr)^2.\end{align}
\indent
The total curvature of the space curve is
\begin{equation}\label{curvature}\int_{\mathbb T} \kappa (x)^2 dx=\int_{\mathbb T}(P^2+Q^2)dx=H_1(P,Q),\end{equation}
\noindent which is an invariant under the flow associated with the NLS.\par
\begin{prop} Let
\begin{equation}\label{area}H_2(P,Q)=-\int_{\mathbb T} P(x )Q'(x )dx.\end{equation}
\indent (i) Then $-H_2$ is convergent almost surely and is invariant under the flow associated with $NLS$,\par
\indent (ii) $-H_2$ represents the area that is enclosed by the contour $\{ u(x): x\in [0, 2\pi ]\}$ in the complex plane, and 
\begin{equation}\qquad \label{meantorsion}H_2={\frac{1}{2}}\int_{\mathbb T} \kappa^2\tau dx;\end{equation}
\indent (iii) $H_2^2\leq 4^{-1}H_1H_3$ for $\beta>0.$
\end{prop}
\begin{proof} (i) The invariance of $H_2$ was noted in \cite{MV1} and can be proved by differentiating through the integral sign and using the canonical equations. We have a series 
$$\int_{\mathbb T} \bar u(\theta ,t ) {\frac{\partial u}{\partial\theta}}(\theta ,t){\frac{d\theta}{2\pi}} =\lim_{N\rightarrow\infty}\sum_{j=-N}^N \overline{\hat u(j)} ij\hat u(j)$$
which converges almost surely. This follows since
\begin{align}\label{max}\int_{B_K} \sup_N\Bigl\vert \sum_{j=-N}^N &\overline{\hat u(j)}ij\hat u(j)\Bigr\vert^p \mu_{K, \beta}(du)\nonumber\\
&\leq \Bigl(\int_{B_K} \Bigl( {\frac{d\mu_{K, \beta}}{dW}}\Bigr)^2 dW\Bigr)^{1/2}\Bigl( \int_{B_K}\sup_N\Bigl\vert \sum_{j=-N}^N \overline{\hat u(j)}ij\hat u(j)\Bigr\vert^{2p}W(du)\Bigr)^{1/2},\end{align}
where the final integral involves the series     
\begin{align}\label{mart}\lim_{N\rightarrow\infty}\sum_{j=-N}^N \overline{\hat u(j)}ij\hat u(j)=\sum_{j=1}^\infty {\frac{\vert z_j\vert^2-\vert z_{-j}\vert^2}{j}}\end{align}
which is a martingale; by Fatou's Lemma, we have
\begin{align}\int_{L^2}\exp \Bigl(\lambda \sum_{j=1}^\infty {\frac{\vert z_j\vert^2-\vert z_{-j}\vert^2}{j}}\Bigr)dW&=\prod_{j=1}^\infty \int_{L^2}\exp \Bigl({\frac{\lambda(\vert z_j\vert^2-\vert z_{-j}\vert^2)}{j}}\Bigr)dW\nonumber\\
&=\Bigl( {\frac{2\pi\lambda}{\sin 2\pi\lambda }}\Bigr)^{1/2}\qquad (-1/2<\lambda <1/2),\end{align}
so the series in (\ref{mart}) is marginally exponentially integrable.
Hence the integrals in (\ref{max}) converge by the $L^p$ martingale maximal theorem for all $1<p<\infty$.\par

\indent (ii) One can write $H_2$ in terms of $P+iQ=\kappa e^{i\sigma}$, and make a change of variables to obtain
$$\kappa ={\frac{\partial (P,Q)}{\partial (\kappa ,\sigma)}}$$
and 
\begin{align*} H_2&={\frac{1}{2}}\int_{\mathbb T}\bigl( P'Q-PQ')dx={\frac{1}{2}}\int_{\mathbb T}\kappa^2\tau dx.\end{align*}
\indent To interpret this as an area, We write $\theta\in [0,2\pi ]$ for the space variable and extend functions on $[0, 2\pi ]$ to harmonic functions on the unit disc via the Poisson kernel. Then by Green's theorem, we can express this invariant in terms of the area of the image of ${\mathbb D}$ under the map to $P+iQ$, as in 
\begin{align}-H_2=\int\!\!\!\int_{\mathbb D}{\frac{\partial (P,Q)}{\partial (x,y)}}dxdy.\end{align}
This is similar to L\'evy's stochastic area, as discussed in Example 5.1 of \cite{IW}. \par

\indent (iii) We then have
$$\Bigl(\int_{\mathbb T}\kappa^2\tau dx \Bigr)^2\leq \int_{\mathbb T} \kappa^2dx\int_{\mathbb T} \kappa^2\tau^2dx$$
which is bounded in terms of other invariants, with $H_2^2\leq 4^{-1}H_1H_3$.\par

\end{proof}
\indent (i) Bourgain interprets $H_2$ in terms of momentum (5.70) \cite{Bo2}.\par
\indent (ii) Suppose that $T\in C^2([0,a]\times [0,b]; {\mathbb S}^2)$, so that $T(x,t)$ represents the spin of the particle at $(x,t)$ and let 
\begin{equation}\label{energy}E(T)=\int_0^a \Bigl\Vert{\frac{\partial T}{\partial x}}(x,t)\Bigr\Vert^2 dx,\end{equation}
which corresponds to our \ref{curvature}. One can consider infinitesimal variations $T\mapsto T+T\times V$ and thereby compute ${\frac{\partial E}{\partial T}}$.
In the focusing case $\beta=-1$, Ding \cite{D} introduces a symplectic structure on the space of such maps such that the Hamiltonian flow is
\begin{equation}\label{Heisenberg} {\frac{\partial T}{\partial t}}=T\times{\frac{\partial^2T}{\partial x^2}}\end{equation}
which corresponds to Heisenberg's equation for the one-dimensional ferro-magnet, and gives the top entry of (\ref{Lax}). There is a a gauge equivalence between the focussing NLS and Heisenberg's ferromagnet. There is also a gauge equivalence between the defocussing NLS and a hyperbolic version of the ferromagnet in which the standard cross product is modified. We have
\begin{align}\Bigl\vert{\frac{\partial^2T}{\partial x^2}}\Bigr\vert^2 &=\Bigl({\frac{\partial \kappa}{\partial x}}\Bigr)^2 +\kappa^2\tau^2+ \kappa^4\\
&=\bigl\vert{\frac{\partial u}{\partial x}}\bigr\vert^2 +\vert u\vert^4.\end{align}
The aim of the next section is to interpret the Lax pair suitably for solutions which are typically not differentiable and for which we have a pair of stochastic differential equations with random matrix coefficients. \par

\section{Gibbs measure transported to the frames}

\indent The compact Lie group $SO(3)$ of real orthogonal matrices with determinant one is a subset of $M_{3\times 3}({\mathbb R})$, which has the scalar product $\langle X,Y\rangle =\trace( XY^\dagger)$ and associated metric $d(X,Y)=\langle X-Y,X-Y\rangle^{1/2} $ such that $\langle XU, YU\rangle =\langle X,Y\rangle$ and $d(XU,YU)=d(X,Y)$ for all $U\in SO(3)$ and $X,Y\in M_{3\times 3}({\mathbb R})$. The Lie group $SO(3)$ has tangent space at the identity element give by the skew symmetric matrices $so(3)$, so the tangent space $T_XSO(3)$ at $X\in SO(3)$ consists of $\{ \Omega X: \Omega\in so(3)\}$, where $so(3)$ is a Lie algebra for $[x,y]=xy-yx$, $x, y \in so(3)$. \par
\indent
With $M_n$ as in (\ref{Mn}), the space $C^\infty(M_n; {\mathbb R})$ is a Poisson algebra for the bracket $\{ f,g\} =\sum_{j=-n}^n {\frac{\partial (f,g)}{\partial (a_j,b_j)}}$, and the canonical equations arise with Hamiltonian $H_3^{(n)}$ on $M_n$. Let $Q$ be the ring of quaternions, and extend the Poisson bracket to $C^\infty (M_n; Q)$ via $\{ f\otimes X, g\otimes Y\}=\{ f,g\}\otimes XY$. Then $({\mathbb R}^3, \times )$ may be realised as $Q/{\mathbb R}I$ and $(so(3), [\cdot ,\cdot ])\cong ({\mathbb R}^3, \times )$; see Example 2.3 of \cite{X}.\par  
\indent 
We consider the differential equation
\begin{equation} \label{skewODE}{\frac{dX}{dt}}=\Omega (t)X; \quad X(0) =X_0\end{equation}
where $t\in [0,1]$ is the evolving time.  Let $\nu_t$ for $t\in [0,1]$ be a weakly continuous family of probability measures on $SO(3)$, such that
\begin{equation}\label{squareintegrability}\int_0^1 \int_{SO(3)}\Vert \Omega \Vert^2_{so(3)}\nu_t(dX)dt<\infty \end{equation}
converges and $\Omega X$ is locally bounded so $X\mapsto \Omega X$ is locally Lipschitz, and such that the family $(\nu_t)$ satisfies the weak continuity equation
\begin{equation}\label{continuity}{\frac{\partial \nu_t}{\partial t}}+\nabla \cdot \Bigl(\Omega X\nu_t\Bigr)=0.\end{equation}
 The weak continuity equation is equivalent to 
\begin{equation} \int_{SO(3)} f(X)\nu_t(dX) =\int_{SO(3)} f(X_t(X_0))\nu_0(dX_0)\end{equation}
for all $f\in C(SO(3); {\mathbb R}),$ where $X_0\mapsto X_t(X_0)$ gives the dependence of the solution of (\ref{skewODE}) on the initial condition.  It is not asserted that the velocity field $\Omega X$ gives an optimal transportation plan, but it does an  upper bound on the Wasserstein transportation cost for the cost $d(X,Y)^2$ on $SO(3)$ of
\begin{equation} {\frac{ W_2(\nu_{t_2}, \nu_{t_1})^2}{t_t-t_1}}\leq \int_{t_1}^{t_2} \int_{so(3)} \Vert \Omega\Vert^2_{so(3)} \nu_t(dX)dt\qquad (0<t_1<t_2<1).\end{equation}
Then by Theorem 23.9 of \cite{V2}, the path $(\nu_t)$ of probability measures is absolutely continuous, so there exists $\ell \in L^1[0,1]$ such that $W_2(\nu_{t_2}, \nu_{t_1})\leq \int_{t_1}^{t_2} \ell (t)dt$ and $1/2$-H\"older continuous, so there exists $C>0$ such that $W_2(\nu_{t_2}, \nu_{t_1})\leq C\vert t_2-t_1\vert^{1/2}$.

\begin{ex} (i) If $\Omega_t\in M_{3\times 3}({\mathbb R})$ is skew, and $X_t,Y_t$ give solutions of the differential equation
\begin{equation} {\frac{dX}{dt}}=\Omega_tX, X(0)=X_0; \quad {\frac{dY}{dt}}=\Omega_tY, Y(0)=Y_0\end{equation}
then $d(X_t,Y_t)=d(X_0,Y_0)$. We deduce that if $X_0$ is distributed according to Haar measure on $SO(3)$, then $X_t$ is also distributed according to Haar measure since the measure, the metric and solutions are all preserved via $X\mapsto XU$. \par
\indent (ii) As an alternative, we can consider $X_0$ to have first column $[0;0;1]$ and observe the evolution of the first column $T$ of $X$ under the (\ref{skewODE}) where $T$ evolves on ${\mathbb S}^2$.   
\end{ex}
\indent We now consider the case in which $\Omega$ is a $so(3)$-valued random variable over $(M_\infty, \mu_{K, \beta}, L^2)$.\par

\begin{prop}\label{GibbsODE}Suppose that $\Omega =\Omega (u( \cdot , t))$ where $u(x,t)$ is a solution of NLS and that 
\begin{equation} \int_{B_K} \Vert \Omega (u(\cdot ,0))\Vert^2_{so(3)}\mu_{K, \beta}(du)\end{equation}
\noindent converges. Then for almost all $u$ with respect to $\mu_{K,\beta}$, there exists a flow $(\nu_t(dX; u))$ of probability measures on $SO(3)$.\end{prop}

\begin{proof} Each solution $u$ of NLS determines $\Omega$ so that the associated ODE (\ref{skewODE}) transports the initial distribution of $X_0\in SO(3)$ to a probability measure on $SO(3)$; then we average over the $u$ with respect to $\mu_K(du)$. This Gibbs measure is invariant under the NLS flow, so by Fubini's theorem
\begin{equation} \int_{B_K} \int_0^1\int_{SO(3)}\Vert \Omega (u(\cdot ,t))\Vert^2_{so(3)}\nu_t(dX)dt\mu_K(du)\end{equation}
converges. Hence the condition (\ref{squareintegrability}) is satisfied, for almost all $u$, and we can invoke Theorem 23.9 of \cite{V2}.\end{proof} 

\indent For the finite-dimensional $M_n$ of (\ref{Mn}) and solutions $u_n=\kappa_ne^{i\sigma_n}$, the modified Hasimoto differential equations are

\begin{equation}\label{SerretFrenetn}{\frac{\partial}{\partial x}}X^{(n)}(x,t) =\begin{bmatrix} 0&\kappa_n &0\\ -\kappa_n&0&\tau_n\\ 0&-\tau_n &0\end{bmatrix}X^{(n)}(x,t),\end{equation}
and

\begin{equation}\label{Laxn}{\frac{\partial}{\partial t}}X^{(n)}(x,t) =\begin{bmatrix} 0&-\tau_n\kappa_n &{\frac{\partial \kappa_n}{\partial x}}\\ \tau_n\kappa_n&0&{\frac{\partial\sigma_n}{\partial t}}-\beta\kappa_n^2\\ -{\frac{\partial \kappa_n}{\partial x}}&-{\frac{\partial\sigma_n}{\partial t}}-\beta\kappa_n^2
&0\end{bmatrix}X^{(n)}(x,t)\end{equation}
involves $\tau_n={\frac{\partial \sigma_n}{\partial x}}$ and $({\frac{\partial \kappa_n}{\partial x}})^2+\tau_n^2\kappa_n^2=({\frac{\partial P_n}{\partial x}})^2+({\frac{\partial Q_n}{\partial x}})^2$ which is continuous, so there exists a solution $X^{(n)}(x,t)\in SO(3)$. 
We can interpret the solutions as elements of a fibre bundle over $(M_n, \mu_K^{(n)}, L^2)$ with fibres that are isomorphic to $SO(3)$.

\indent Let $P+iQ=\kappa e^{i\sigma}$ be a solution of NLS and let
\begin{equation}\label{SerretFrenet1}
\Omega_1=\begin{bmatrix}0&\kappa&0\\ -\kappa &0&\tau\\ 0&-\tau   &0\end{bmatrix}.\end{equation}

\begin{prop} (i) Let $P+iQ=\kappa e^{i\sigma}$ be a solution of NLS with initial data in $P(x,0)+iQ(x,0)\in B_K\cap H^1$. Then $\Omega_1$ in (\ref{SerretFrenet1}) gives an $so(3)$-valued vector field in $L^2(\kappa^2 (x,t)dx)$.\par
(ii) Let $P+iQ=\kappa e^{i\sigma}$ be a solution of NLS with initial data 
   $P(x,0)+iQ(x,0)\in H^1\cap B_K$, and let $P_n+iQ_n=\kappa_ne^{i\sigma_n}$ be the corresponding solution of the NLS truncated in Fourier space, giving matrix $\Omega_1^{(n)}$.
Let $X^{(n)}_t(x)$ be a solution of (\ref{SerretFrenetn}) and suppose that $X^{(n)}$ converges weakly in $L^2$ to $X_t(x)$. Then $X_t$ gives a weak solution of (\ref{SerretFrenet}).\end{prop}
\begin{proof}  
(i) With $\omega=\sqrt{\kappa^2+\tau^2}$, we have
$$\exp (h\Omega_1)=I+{\frac{\sin h\omega}{\omega}} \Omega_1+{\frac{1-\cos h\omega}{\omega^2}}\Omega_1^2$$
where the entries of $\Omega_1^2$ are bounded by $\kappa^2+\tau^2$, hence
\begin{equation}\label{Omega1L2}\int_{\mathbb T}\Vert \Omega_1(x,t)\Vert^2\kappa (x,t)^2dx<\infty\end{equation}
for $u\in H^1$; however, there is no reason to suppose that $\tau$ itself is integrable with respect to $dx$. \par
(ii) By (\ref{acceleration}) and (\ref{curvature}), we have $\kappa\Omega_1\in L^2_x$ for all $u\in H^1$. Moreover, 
Bourgain \cite{Bo} has shown that for initial data $P(x,0)+iQ(x,0)=\kappa (x,0)e^{i\sigma (x,0)}$ in $H^1\cap B_K$, the map 
\begin{equation} \kappa (x,0)e^{i\sigma (x,0)}\mapsto \kappa (x,t)\Omega_1(x,t)\in L^2\end{equation}
is Lipschitz continuous for $0\leq t\leq t_0$ with Lipschitz constant depending upon $t_0, K>0$. We have
\begin{align} {\frac{\Vert \kappa (x+h,t)X(x+h,t)-\kappa (x,t)X(x,t)\Vert^2}{h^2}}&\leq 2\Bigl({\frac{1}{h}}\int_x^{x+h} \bigl\vert {\frac{\partial \kappa}{\partial y}}(y,t)\bigr\vert dy\Bigr)^2\nonumber\\
&\quad+2\Bigl({\frac{1}{h}}\int_x^{x+h} \kappa (y,t)\Vert \Omega_1(y,t)\Vert dy\Bigr)^2\end{align}
where the right-hand side is integrable with respect to $x$ by the Hardy--Littlewood maximal inequality and (\ref{Omega1L2}). 
Suppose that $X^{(n)}$ is a solution of (\ref{SerretFrenet}). We take $\tau_n$ to be locally bounded. Then by applying Cauchy--Schwarz inequality to the integral
$$ X^{(n)}(x+h,t)-X^{(n)}(x,t)=\int_0^h \Omega_1^{(n)} (x+s,t)X^{(n)}(x+s,t)ds,$$
we deduce that
\begin{align} \int_{[0,2\pi]} &\Vert X^{(n)}(x+s,t)-X^{(n)}(x,t)\Vert_{M_{3\times 3}({\mathbb R})}^2\kappa_n (x,t)^2dx\nonumber\\
&\leq h\int_0^h \int_{[0, 2\pi ]} \Vert \Omega_1^{(n)}(x+s,t)\Vert^2_{M_{3\times 3}({\mathbb R})} \kappa_n (x,t)^2dxds\end{align}
where the integral is finite by (\ref{Omega1L2}).\par
Also
\begin{align*}\sum_{j=1}^N {\frac{\Vert X^{(n)}(x_{j},t)-X^{(n)}(x_{j-1},t)
\Vert^2_{M_{3\times 3}({\mathbb R})}}{x_j-x_{j-1}}}\leq\int_{x_0}^{x_N}\Vert\Omega_1^{(n)}(x,t)\Vert^2dx\end{align*}
for $0<x_1<x_2<\dots <x_N<2\pi$. 

\par
We have
\begin{equation} {\frac{\partial}{\partial x}}\bigl(\kappa_nX^{(n)}\bigr) ={\frac{\partial \kappa_n}{\partial x}}X^{(n)}+\kappa^{(n)} \Omega_1^{(n)} X^{(n)}\end{equation}
so for $Z\in C^\infty ([0, 2\pi ]; M_{3\times 3}({\mathbb R}))$ and the inner product on $M_{3\times 3}({\mathbb R})$, we have
\begin{align}\langle & \kappa_n(2\pi )X^{(n)}(2\pi ), Z(2\pi )\rangle -
\langle \kappa_n(0)X^{(n)}(0), Z(0 )\rangle-\int_0^{2\pi} \kappa_n(x)\langle X^{(n)}(x), Z(x)\rangle\, dx\nonumber\\ 
&= \int_0^{2\pi} {\frac{\partial \kappa_n}{\partial x}}\langle X^{(n)}(x), Z(x)\rangle\, dx+ \int_0^{2\pi} \langle X^{(n)}, \kappa_n (x) \Omega_1^{(n)}(x)^TZ(x)\rangle dx\end{align}
where $\kappa_n\rightarrow \kappa$ in $H^1$, so with norm convergence, we have ${\frac{\partial\kappa_n}{\partial x}}\rightarrow {\frac{\partial \kappa}{\partial x}}$ in $L^2$, and $\kappa_n \Omega^{(n)}\rightarrow \kappa\Omega_1$ as $n\rightarrow\infty$, and with weak convergence in $L^2$, we have $X^{(n)}\rightarrow X$, so 
\begin{align}\langle \kappa (2\pi )X (2\pi ), Z(2\pi )\rangle &-
\langle \kappa (0)X (0), Z(0 )\rangle-\int_0^{2\pi} \kappa (x)\langle X(x), Z(x)\rangle\, dx\nonumber\\ 
&= \int_0^{2\pi} {\frac{\partial \kappa}{\partial x}}\langle X(x), Z(x)\rangle\, dx+ \int_0^{2\pi} \langle X, \kappa (x) \Omega_1(x)^TZ(x)\rangle dx.\end{align}

\end{proof}
The simulation of this differential equation computes $X_x\in {\mathbb S}^2$ starting with $X_0=[0;0;1]$ and produces a frame $\{ X_x, \Omega_xX_x, X_x 	\times \Omega_xX_x\}$ of orthogonal vectors. Geodesics on ${\mathbb S}^2$ are the curves such that the principal normal is parallel to the position vector, namely the great circles. For a geodesic, $X_x\times \Omega_xX_x$ is perpendicular to the plane that contains the great circle.

\indent Let $P+iQ=\kappa e^{i\sigma}$ be a solution of NLS and let
\begin{equation}
\Omega_2 =\begin{bmatrix}0&-\kappa\tau&{\frac{\partial \kappa}{\partial x}}\\ \kappa\tau &0&0\\ -{\frac{\partial \kappa}{\partial x}}&0  &0\end{bmatrix}.\end{equation}
\begin{prop} (i) Let $P+iQ=\kappa e^{i\sigma}$ be a solution of NLS with initial data 
   $P(x,0)+iQ(x,0)\in B_K$.  Then $x\mapsto \int_0^x\Omega_2(y,t)dy$ gives a $so(3)$-valued stochastic of finite quadratic variation on $[0, 2\pi ]$ almost surely with respect to $\mu_K(dPdQ)$.\par 
  (ii)  Let $P+iQ=\kappa e^{i\sigma}$ be a solution of NLS with initial data
   $P(x,0)+iQ(x,0)\in H^1\cap B_K$, and let $P_n+iQ_n=\kappa_ne^{i\sigma_n}$ be the corresponding solution of the NLS truncated in Fourier space, giving matrix $\Omega_2^{(n)}$.
Let $X^{(n)}_t$ be a solution of \ref{Laxn}). Then $X_t^{(n)}$ converges in $L^2_x$ norm to $X_t$ as $n\rightarrow\infty$ where $X_t$ gives a weak solution of (\ref{Lax}).\end{prop}
\begin{proof} (i) The essential estimate is
\begin{align}\int_{B_K} \sum_j &\vert \kappa (x_{j+1},t)-\kappa (x_j,t)\vert^2 \mu_K(du)\nonumber\\
&\leq \sum_{j}\Bigl( \int_{B_K}\vert u(x_{j+1},t)-u(x_j,t)\vert^2\mu_K(du)\Bigr)\nonumber\\
&\leq \sum_{j}\Bigl( \int_{B_K}\vert u(x_{j+1},t)-u(x_j,t)\vert^4W_K(du)\Bigr)^{1/2}\Bigl(\int_{B_K}\Bigl( {\frac{d\mu_K}{dW}}\Bigr)^2dW\Bigr)^{1/2}\nonumber\\
&\leq C\sum_{j}\Bigl( \int_{B_K}\vert u(x_{j+1},t)-u(x_j,t)\vert^2W(du)\Bigr)^{1/2}\nonumber\\
&\leq C\sum_j (x_{j+1}-x_j)\leq 2\pi C.\end{align}
Also, we can control the $\kappa\tau$ term via
$$\int_0^x \kappa d\sigma-2^{-1}\kappa^2d\sigma\wedge d\sigma=\int_0^x \kappa\nabla\sigma \cdot\begin{bmatrix} dP\\ dQ\end{bmatrix}=\int_0^x {\frac{-QdP+PdQ}{\sqrt{P^2+Q^2}}}$$
which is a bounded martingale transform of Wiener loop.\par
\indent (ii)  By (\ref{acceleration}) and (\ref{curvature}), we have $\Omega_2\in L^2_x$ for all $u\in H^1$. 
Bourgain \cite{Bo} has shown that for initial data $P(x,0)+iQ(x,0)=\kappa (x,0)e^{i\sigma (x,0)}$ in $H^1\cap B_K$, the map 
\begin{equation} \kappa (x,0)e^{i\sigma (x,0)}\mapsto \Omega_2(x,t)\in L_x^2\end{equation}
is Lipschitz continuous for $0\leq t\leq t_0$ with Lipschitz constant depending upon $t_0, K>0$. 
We have
$$\int_0^{2\pi}\Vert \Omega_2 (x)\Vert^2 dx\leq 2 \int_0^{2\pi}\Bigl( \Bigl({\frac{\partial \kappa}{\partial x}}\Bigr)^2+\kappa (x)^2\tau (x)^2 +\kappa (x)^4\Bigr) dx,$$
where the final integral is part of the Hamiltonian. With $Z\in C^\infty ({\mathbb T}; M_{3\times 3}({\mathbb R}))$, we have the integral equation for the pairing $\langle \cdot ,\cdot\rangle$ on $L^2([0, 2\pi ], M_{3\times 3}({\mathbb R}))$
\begin{equation} \langle X^{(n)}_t ,Z\rangle =\langle X_0^{(n)}, Z\rangle +\int_0^t \langle X^{(n)}_s , (\Omega^{(n)}_s)^TZ\rangle\,ds. \end{equation}

Consider the variational differential equation in $L^2([0, 2\pi ], M_{3\times 3}({\mathbb R}))$
\begin{align} {\frac{d}{dt}}( X^{(m)}(x,t)-X^{(n)}(x,t))&=\Omega_2^{(n)}(x,t)(X^{(m)}(x,t)-X^{(n)}(x,t))\nonumber\\
&\quad+(\Omega_2^{(m)}(x,t)-\Omega_2^{(n)}(x,t)) X^{(m)}(x,t)\end{align}
where $\Omega_2^{(n)}(x,t)$ and $\Omega_2^{(m)}(x,t) -\Omega_2^{(n)}(x,t)$ are skew.

We introduce a family of matrices $U^{(n)}(x;t,s)$ such that $U^{(n)}(x;t,r)U^{(n)}(x;r,s)=U^{(n)}(x;t,s)$ for $t>r>s$ and $U^{(n)}(x;t,t)=I$ such that \begin{equation}{\frac{\partial}{\partial t}}U^{(n)}(x;t,s)=\Omega_2^{(n)} (x;t)U^{(n)}(x;t,s).\end{equation} Then the variational equation has solution
\begin{align*} X^{(m)}(x,t)-X^{(n)}(x,t)&=U^{(n)}(x;t,0)( X^{(m)}(x,0)-X^{(n)}(x,0))\\
&\quad +\int_0^t U^{(n)}(x;t,r) (\Omega_2^{(m)} (x;r)-\Omega_2^{(n)}(x;r)) X^{(m)}(x,r)dr.\end{align*}
Then 
\begin{align} {\frac{d}{dt}}&\bigl\langle X^{(m)}(t)-X^{(n)}(t),X^{(m)}(t)-X^{(n)}(t)\rangle_{L^2_x} \nonumber\\
&=2\Re \bigl\langle (\Omega_2^{(m)}(t)-\Omega_2^{(n)}(t)) X^{(m)}(t),X^{(m)}(t)-X^{(n)}(t)\bigr\rangle_{L^2_x}\nonumber\\
&\leq \Vert \Omega_2^{(m)}(t)-\Omega_2^{(n)}(t)\Vert^2_{L^2_x}\Vert X^{(m)}(t)\Vert^2_{L^2_x} +\Vert X^{(m)}(t)-X^{(n)}(t)\Vert^2_{L^2_x}\end{align}
so from this differential inequality we have
\begin{equation}\label{intinequal}\Vert X^{(m)}(t)-X^{(n)}(t)\Vert_{L^2_x}^2\leq e^{t}\Vert  X^{(m)}(0)-X^{(n)}(0)\Vert_{L^2_x}^2+\int_0^t e^{t-s}\Vert \Omega_2^{(m)}(s)-\Omega_2^{(n)}(s)\Vert_{L^2_x}^2ds. \end{equation}  
Now $X^{(m)}(0)-X^{(n)}(0)\rightarrow 0$ and $\Omega_2^{(m)}(s)-\Omega_2^{(n)}(s)\rightarrow 0$ in $L_x^2$ norm as $n,m\rightarrow\infty$, so there exists $X(x,t)\in L_x^2$ such that $X(x,t)-X^{(n)}(x,t)\rightarrow 0$ in $L^2_x$ norm as $n\rightarrow\infty$.\par 
 We deduce that 
\begin{equation} \langle X(t) ,Z\rangle_{L^2_x} =\langle X_0, Z\rangle_{L^2_x} +\int_0^t \langle X_u , (\Omega_2(u))^TZ\rangle_{L^2_x}\,du,\end{equation}
so we have a weak solution of the ODE.
\end{proof}
Let $\Omega^{(n,u_n)}_2(x,t)$ be the Fourier truncated matrix that corresponds to a solution $u_n$ of the Fourier truncated equation $NLS_n$, then let $X^{n,u_n}(x,t)$ be the solution of the ODE (\ref{Laxn}). By Proposition \ref{GibbsODE}, the map $u_n\mapsto X^{n,u_n}(\cdot , t)$ pushes forward the modified Gibbs measure $\mu_K^{(n)}$ to a measure on $(C(M_n;SO(3)), L^2)$ that satisfies a Gaussian concentration of measure inequality with constant $\alpha (\beta ,K)/n^2$; compare (\ref{logsob}).
\begin{cor}
For each $Z\in L^2([0, 2\pi ]; M_{3\times 3}({\mathbb R}))$, introduce the ${\mathbb R}$-valued random variable on $(M_n, L^2, \mu_K^{(n)})$ by
\begin{equation} Z_n(u_n)=\int_{[0, 2\pi ]} \langle X^{(n,u_n)}(x,t), Z(x)\rangle dx. \end{equation}
 (i) Then the distribution $\nu^{(n)}$  of $Z_n$ satisfies the Gaussian concentration inequality
\begin{equation} \label{Zconcentration}
\int_{M_n} \exp \Bigl( tZ_n -t\int_{M_n}Z_nd\mu_K^{(n)}\Bigr) \mu_K^{(n)} (du_n)\leq \exp \bigl( n^2t^2/\alpha (\beta ,K)\bigr)\qquad (t\in {\mathbb R}).\end{equation} 
(ii) Let $\nu_N^{(n)}=N^{-1}\sum_{j=1}^N\delta_{Z_n^{(j)}}$ be the empirical distribution of $N$ independent copies of $Z_n$. Then $W_1(\nu_N^{(n)}, \nu^{(n)})\rightarrow 0$ almost surely as $N\rightarrow\infty$. 
\end{cor}
\begin{proof} (i) As with $u_n$, we introduce the corresponding data for another solution $v_n$. As in (\ref{intinequal}), we have
\begin{align}\label{groninequal}\Vert X^{(n,u_n)}(x,t)-X^{(n, v_n)}(x,t)\Vert_{{\mathbb R}^3}^2&\leq e^{t}\Vert  X^{(n,u_n)}(0)-X^{(n,v_n)}(0)\Vert_{{\mathbb R}^3}^2\nonumber\\
&\quad +\int_0^t e^{t-s}\Vert \Omega_2^{(n,u_n)}(x, s)-\Omega_2^{(n,v_n)}(x,s)\Vert_{so(3)}^2ds. \end{align}
For given initial condition $X^{n,v_n)}(0)=X^{(n,u_n)}(0)$, and $T>0$, we can take the supremum over $t$, then integrate this with respect to $x$ and obtain 
\begin{align}\int_0^{2\pi}\sup_{0<t<T}\Vert X^{(n,u_n)}(x,t)&-X^{(n,v_n)}(x,t)\Vert_{{\mathbb R}^3}^2dx\nonumber\\
&\leq e^T\int_0^T 
\Vert \Omega_2^{(n,u_n)}(x,s)-\Omega_2^{(n,v_n)}(x,s)\Vert_{L^2_x}^2ds\end{align}
so $\Omega^{(u)}\mapsto X^u$ is a Lipschitz function $L^2 ([0, 2\pi ]\times [0, T], so(3))\rightarrow L^2([0,2\pi]; L^\infty ([0,T],{\mathbb R}^3))$. 
By Bourgain's results, there exists $C>0$ such that
\begin{align} \Vert \Omega_2^{(n,u_n)}(x,s)-\Omega_2^{(n,v_n)}(x,s)\Vert_{L^2_x}&\leq C\Vert u_n(x,s)-v_n(x,s)\Vert_{H_x^1}\nonumber\\
&\leq Cn\Vert u_n(x,0)-v_n(x,0)\Vert_{L^2_x}, \end{align}
so $u_n\mapsto X^{(n,u_n)}$ is a Lipschitz function on $L^2_x$, albeit with a constant growing with $n$. Thus we can push forward the modified Gibbs measure $(M_n,L^2, \mu_{K, \beta}^{(n)})\rightarrow L^2([0, 2\pi ]; M_{3\times 3}({\mathbb R}))$ so that the image measure satisfies a Gaussian concentration inequality with constant $\alpha (\beta ,K)/n^2$ dependent upon $n$.
For each $Z\in L^2([0, 2\pi ]; M_{3\times 3}({\mathbb R}))$, we introduce $Z_n$, so that
where $u_n\mapsto Z_n$ is $Cn$-Lipschitz function from $(M_n,L^2, \mu_{K, \beta}^{(n)})$ to ${\mathbb R}$. The random variable $Z_n$ therefore satisfies the Gaussian concentration inequality (\ref{Zconcentration}).\par
\indent (ii) By Theorem \ref{sanov}, we can use the Borel--Cantelli Lemma to show that  
$${\mathbb P}\Bigl[ \bigl\vert W_1(\nu_N^{(n)}, \nu^{(n)})-{\mathbb E}W_1(\nu_N^{(n)}, \nu^{(n)})\bigr\vert >\varepsilon \quad{\hbox{for infinitely many}}\quad N\Bigr]=0\qquad (\varepsilon >0),$$
where by Proposition \ref{prop:fluctuation}, ${\mathbb E}W_1(\nu_N^{(n)}, \nu^{(n)})\rightarrow 0$ as $N\rightarrow\infty$.
\end{proof}

Consider a coupling of $(M_n, L^2, \mu^{(n)}_{K, \beta})$ and $(M_\infty , L^2, \mu_{K, \beta})$ involving measure $\pi_n$. For any bounded continuous $\varphi :{\mathbb C}\rightarrow {\mathbb R}$ we can consider

\begin{align}\label{weakdiffer}\int_{M_n} \varphi (Z_n(u_n)) \mu^{(n)}_{K, \beta}(du_n)-&\int_{M_\infty }\varphi (Z(u))\mu_{K, \beta}(du)\nonumber\\
&=\int\!\!\!\int_{M_n\times M_\infty}\bigl( \varphi (Z_n(u_n))-\varphi (Z(u))\bigr)\pi_n(du_ndu)\end{align}
where
\begin{align*} {\mathcal D}_{L^2} (\hat M_n, \hat M_\infty )^2=\int\!\!\!\int_{M_n\times M_\infty} \delta (u_n, u)^2 \pi_n(du_ndu)\rightarrow 0\qquad (n\rightarrow\infty ).\end{align*}

\begin{prop} Let $(\varphi_j)_{j=1}^\infty$ be a dense sequence in $Ball (C_c({\mathbb C}; {\mathbb R}))$ and $(Y_\ell  )_{\ell =1}^\infty$ a dense sequence in $Ball (L^2)$. Then there exists a subsequence $(n_k)$ such that 
\begin{equation}\int_{M_{n_k}} \varphi_j\bigl(\langle X^{(n_k, u_{n_k})} , Y_\ell \rangle \bigr) \mu_{K, \beta}^{(n_k)} (du_{n_k})\end{equation}
converges as $n_k\rightarrow\infty$ for all $j,\ell \in {\mathbb N}$. 
\end{prop}
\begin{proof} We can introduce a metric so that $M=\prod_{n=1}^\infty M_n\sqcup M_\infty$ becomes a complete and separable metric space, and we can transport $\mu_n$ onto $M$. Then $\omega =2^{-1}\mu_\infty+\sum_{n=1}^\infty 2^{-n-1}\mu_n$ is a probability measure on $M$, and $\mu_n$ is absolutely continuous with respect to $\omega$, so $d\mu_n=f_nd\omega$ for some probability density function $f_n\in L^1(\omega )$. 
By convergence in total variation from Lemma \ref{Mmetricconvergence} (ii), here exists $f_\infty\in L^1(\omega )$ such that $f_n\rightarrow f_\infty$ in $L^1$ as $n\rightarrow\infty$. Given a bounded sequence $(g_n)_{n=1}^\infty$ in $L^\infty (\omega )$, there exists $g_\infty\in L^\infty (\omega )$ and a subsequence $(n_k)$ such that
\begin{equation} \int g_{n_k}d\mu_{n_k}=\int g_{n_k}f_{n_k}d\omega \rightarrow \int g_\infty f_\infty d\omega =\int g_\infty d\mu_\infty.\end{equation} 
\end{proof}

\begin{rem} 
For $u\in M_\infty,$ we have $u_n=D_nu\in M_n$ so that $u_n\rightarrow u$ in $L^2$ norm as $n\rightarrow\infty$. It is plausible that (\ref{weakdiffer}) tends to $0$ as $n\rightarrow\infty$, but we do not have a proof. Unfortunately, the constants are not sharp enough to allow us to use (\ref{L2convergence}) to deduce $W_2$ convergence for the distributions on $SO(3)$.\par

\end{rem}
\section{Experimental Results}
Our objective in this section is to obtain a (random) numerical approximation to the solution of \eqref{SerretFrenetn}. 
We consider the case where the parameter $\beta$ in \eqref{NLS} is equal to $0$. Note that in this case, the Gibbs measure reduces to Wiener loop measure and stochastic processes with the Wiener loop measure as their law are by definition Brownian loop. 
 Equation \eqref{SerretFrenetn} is a PDE with respect to the space variable $x$, while the parameter of a stochastic process in an SDE is colloquially referred to as time. To avoid confusion, in this section we refer to $x$ as $s$; whereas the time variable $t$ is suppressed.

Recall the polar decomposition $P+iQ=\kappa e^{i\sigma}$ where, $\kappa=\sqrt{P^2+Q^2}$ and $\sigma$ is such that $\tau={\frac{\partial \sigma}{\partial s}}$.
Define $\sigma_\epsilon(P,Q) := \tan^{-1}(\frac{PQ}{P^2 + \epsilon^2})$ as the regularised Ito integral of $\tau$. The Ito differential  $d\sigma_{\epsilon}$ can be written as $d\sigma_{\epsilon} = f_1(P,Q) dP + f_2(P,Q)dQ + f_3(P,Q)ds$, where 
\begin{align*}
    &f_1(P,Q):= \frac{(\epsilon^2 - P^2)Q}{(\epsilon^2 + P^2)^2 + P^2Q^2}, \\ 
    &f_2(P,Q):= \frac{ P(\epsilon^2 + P^2)}{(\epsilon^2 + P^2)^2 + P^2Q^2}, \\ 
    &f_3(P,Q):=-\frac{2P^3Q(\epsilon^2 + P^2)}{((\epsilon^2 + P^2)^2 + P^2 Q^2 )^2}-\frac{2PQ\left((\epsilon^2 + P^2)^2 + P^2Q^2 \right)}{\left((\epsilon^2 + P^2)^2 + P^2Q^2 \right)^2}\\
    &\qquad\qquad\qquad-\frac{(\epsilon^2 - P^2)Q\left(2PQ^2 + 4P(\epsilon^2 + P^2) \right)}{\left((\epsilon^2 + P^2)^2 + P^2Q^2 \right)^2}
    .
\end{align*}

We can write \eqref{SerretFrenetn} in the form of a SDE, including a correction to convert from a Stranovich SDE into an Ito SDE as follows
\begin{align}
    dX_s = &\mathbf{A} X_s ds + \mathbf{B}  X_s  dP + \mathbf{C} X_s  dQ \label{equation:SDE1}
    \end{align}
    where,
    \begin{align}
& \mathbf{A} = \begin{bmatrix}
0 & \sqrt{P^2 + Q^2} & 0 \\
-\sqrt{P^2 +Q^2} &  \frac{1}{2}f_1^2(P,Q) + \frac{1}{2}f_2^2(P,Q) & f_3(P,Q) \\
0 & -f_3(P,Q) &  \frac{1}{2}f_1^2(P,Q) + \frac{1}{2}f_2^2(P,Q) 
\end{bmatrix}, \nonumber \\
&\mathbf{B} = \begin{bmatrix}
0 & 0 & 0 \\
0 & 0 & f_1(P,Q)\\
0 & -f_1(P,Q) & 0 
\end{bmatrix}, \quad
\mathbf{C} = \begin{bmatrix}
0 & 0 & 0 \\
0 & 0 & f_2(P,Q) \\
0 & -f_2(P,Q) & 0 
\end{bmatrix}.\label{equation:SDE2}
\end{align}

As justified above $P$ and $Q$ are each a Brownian bridge with period $T = 2\pi$, thus they can be expressed in terms of Brownian motions $W_1$ and $W_2$; that is, $P(s) = W_1(s) - sW_1(2\pi)/2\pi$ and likewise for $Q$. Equation \eqref{equation:SDE1} is now written as a standard Ito SDE,

\begin{align}
    dX_s = &\Bigl(\mathbf{A} + \frac{W_1(2\pi)}{2\pi} \mathbf{B} + \frac{W_2(2\pi)}{2\pi} \mathbf{C}\Bigr) X_s ds + \mathbf{B}  X_s  dW_1 + \mathbf{C} X_s  dW_2 \label{equation:SDE3}
\end{align}
where $\mathbf{A}$, $\mathbf{B}$ and $\mathbf{C}$ are defined as in Equation \eqref{equation:SDE2}. The resulting stochastic process $X_s \in SO(3)$ is then used to rotate the unit vector $y_0 = [ 0, 0, 1]^T$ on $\mathbb{S}^2$ by $y_s = X_s y_0$. The sample paths of this process can be described by construction of a frame $\lbrace y_s, y^\prime_s, y_s \times y_s^\prime \rbrace$. In order to simulate this SDE, we make use of a 
numerical scheme for matrix SDEs in $SO(3)$ developed by Marjanovic and Solo  \cite{MS}. This involves a single step geometric Euler-Maruyama method, called g-EM, in the associated Lie algebra.  Figure \ref{fig:sample_path} demonstrates a sample-path of $y_s$ generated via this method. 
\begin{figure}
    \centering
    \includegraphics[width=10cm]{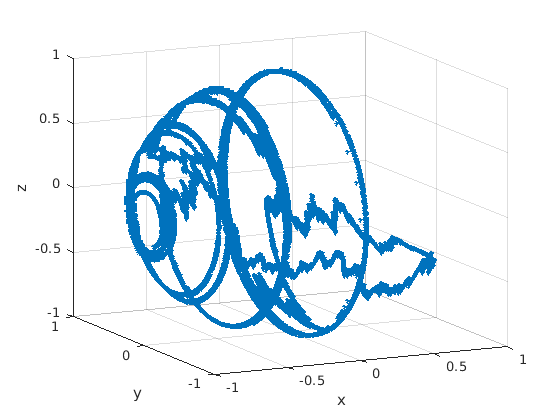}
    \caption{One sample path of the stochastic process $X_s$, over $[0,10]$ with step size of $h = 10^{-5}$.}
    \label{fig:sample_path}
\end{figure}
The sample paths start off on the great circle perpendicular to the y-axis, and so have constant binormal $y_s \times y_s^\prime$. As a sample path extends past the great circle the binormal vector at each point deviates slowly, thus a sample path can be thought of as a precessing orbit. 

The Ito process $y_s$ is derived from the solution to Equation \eqref{equation:SDE3} and takes values in $\mathbb{R}^3$. Let $\hat{y}_{s,h}$ denote the numerical approximation to $y_s$ on $[0,T]$ with step size $h$, which is calculated using the g-EM method. The approximation error converges to zero in the $L^2$ space of Ito processes as the step size $h \rightarrow 0$,
\begin{equation}
    \mathbb{E} \left[ \sup_{0 \leq s < T}  \| y_s - \hat{y}_{s, h} \|_{\mathbb{R}^3}^2 \right] = \mathcal{O}(h^{1-\varepsilon}),
\end{equation}
for some $\varepsilon > 0$ \cite{PS}. A value of $\varepsilon = 1/4$ allows us to maintain control of the implied constants on the interval $[0,1]$, and $h$ is taken to be $10^{-5}$. We apply g-EM to Equation \eqref{equation:SDE3} on the interval $[0,10]$, upon which a smaller value of $h$ would be welcomed. However, we are attempting to calculate a distribution, so we need a large number of sample-paths.\par 
\indent The computational complexity of simulating a single sample-path is $\mathcal{O}(T/h)$ where $T$ denotes the length of the interval simulated. Therefore, for a total of $N$ samples, the computational complexity of our simulation algorithm is $\mathcal{O}(NT/h)$. We run our simulations using a machine equipped with an 8-core Intel Xeon Gold 6248R CPU with a clock speed of $2993$ Mhz; we take advantage of integrated parallelisation in MATLAB. With $h = 10^{-5}$ and $N = 2 \times 10^6$ the algorithm takes around $1$ week to run on our system.

Since the sample paths are constrained to $\mathbb{S}^2$ the points $y_s$ can be specified in spherical coordinates of longitude $\theta_s \in [-\pi,\pi)$ and colatitude $\phi_s \in [0, \pi]$. Figure \ref{fig:bivariate_at_s1} demonstrates the empirical joint distribution of $\theta_s$ and $\phi_s$ for two different values of $s$. 
\begin{figure}
    \centering
    \includegraphics[width=75mm]{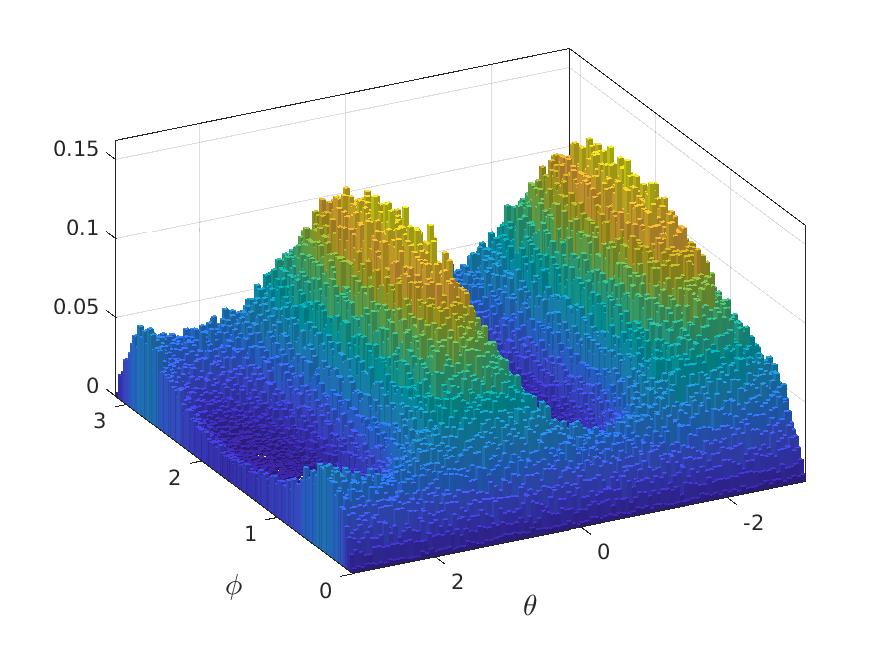}
    \includegraphics[width=75mm]{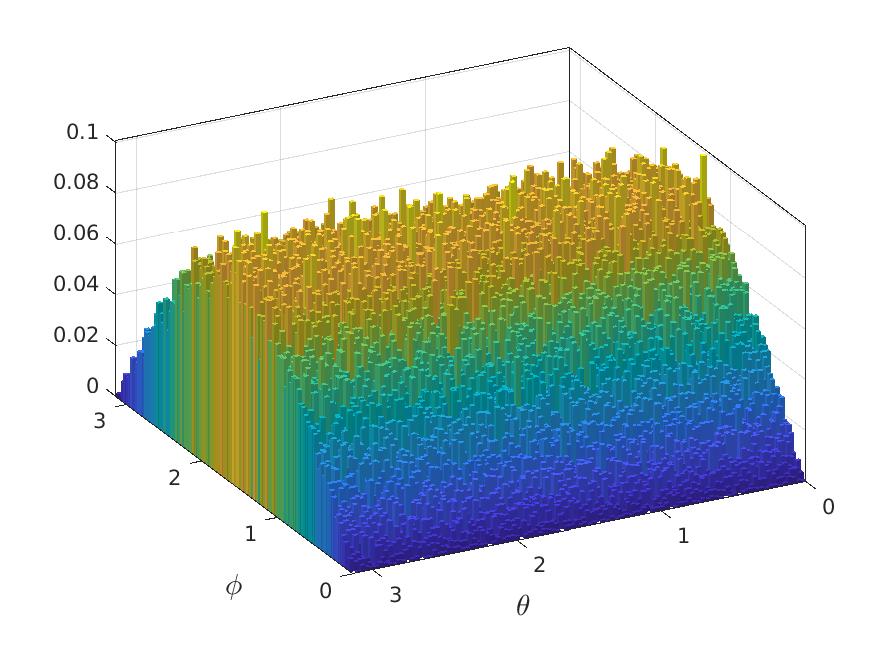}
    \caption{\textbf{Left:} A histogram of $(\theta_s, \phi_s)$ over one hemisphere at $s = 1$. \textbf{Right:} A histogram of $(\theta_s, \phi_s)$ over the sphere at $s = 10$.}
    \label{fig:bivariate_at_s1}
\end{figure}
As can be observed, the distribution of $(\theta_s, \phi_s)$ varies with $s$. We hypothesise that the angles $\theta_s$ and $\phi_s$ evolve to become statistically independent, and that $y_s$ will  eventually be uniformly distributed on the sphere.  In the remainder of the section, we test this hypothesis statistically.
\vskip.1in

\paragraph{\textbf{Wasserstein distance between measures on $\mathbb S^2$.}} We start by calculating the Wasserstein distance $W_1(\nu_1 ,\nu_2 )$ between probability measures $\nu_1$ and $\nu_2 $ on ${\mathbb S}^2$, which are absolutely continuous with respect to area and have disintegrations
\begin{align*} d\nu_j &=f_j(\theta) g_j(\phi\mid \theta ) \sin\phi \, d\phi d\theta\qquad (\theta\in [-\pi,\pi], \phi\in [0, \pi ], j=1,2)\end{align*}
where $f_j$ $(j=1,2)$ are probability density functions on $[-\pi , \pi ]$ that give the marginal distributions of $\nu_j$ in the longitude $\theta$ variable, and $g_j$ in the colatitude variable. Let $F_j$ be the cumulative distribution function of $f_j(\theta )d\theta$ and $G_j$ be the cumulative distribution function of $g_j(\phi )\sin\phi\,d\phi$. We measure $W_1(\nu_1 ,\nu_2 )$ in terms of one-dimensional distributions. 
Given distributions on ${\mathbb R}$ with cumulative distribution functions $F_1$ and $F_2$, we write $W_1(F_1,F_2)$ for the Wasserstein distance between the distributions for cost function $\vert x-y\vert$. Let $\psi :[-\pi , \pi ]\rightarrow [-\pi, \pi ]$ be an increasing function that induces $f_2(\theta) d\theta$ from $f_1(\theta )d\theta$; then
$$W_1( \nu_1 , \nu_2 )\leq W_1\bigl(F_1, F_2)+\int_{-\pi}^{\pi} W_1\bigl(G_2(\cdot \mid \psi (\theta )),G_1(\cdot \mid \theta )\bigr)f_1(\theta )d\theta .$$
In particular, for $f_1(\theta) =1 /(2\pi )$ and $g_1(\phi )=1/2$, we have a product measure $\nu_1 (d\theta d\phi )= (4\pi)^{-1}\sin\phi d\phi d\theta$ giving normalized surface area on the sphere. Then $F_1(\theta )=(\theta+\pi)/(2\pi )$ and $F_2(\psi(\theta ))=(\theta+\pi)/(2\pi) $, so $\psi (2\pi (\tau -1/2) )$  for $\tau\in [0,1]$ gives the inverse function of $F_2$. We deduce that 
\begin{equation} W_1(F_1,F_2)=\int_{-\pi}^{\pi} \Bigl\vert {\frac{\theta+\pi}{2\pi}}-F_2(\theta )\Bigr\vert d\theta\end{equation}
and 
\begin{equation} W_1\bigl(G_2(\cdot \mid \psi (\theta )),G_1(\cdot \mid \theta )\bigr)=\int_0^\pi \Bigl\vert \int_0^\phi  \bigl(g_2(\phi'\mid \psi (\theta ))-(1/2)\bigr)\sin\phi'd\phi'\Bigr\vert d\phi\end{equation}
\vskip.05in
Hence the Wasserstein distance can be bounded in terms of the cumulative distribution functions by
\begin{align} W_1( \nu_1 , \nu_2 )&\leq W_1(F_2, F_1)+W_1(G_2,G_1)+\int_{-\pi}^{\pi}W_1(G_2(\cdot \mid \theta ),G_2)dF_1(\theta) \nonumber\\
&=\int_{-\pi}^{\pi}\Bigl\vert {\frac{\theta+\pi}{2\pi}}-F_2(\theta )\Bigr\vert d\theta
+\int_0^\pi \Bigl\vert G_2(\phi )-{\frac{1-\cos\phi }{2}}\Bigr\vert d\phi\nonumber\\
&\quad +\int_{-\pi}^{\pi}\int_0^\pi \bigl\vert G_2(\phi \mid\theta )-G_2(\phi )\bigr\vert dF_1(\theta )d\phi
\end{align}
where we have used the triangle inequality to obtain a more symmetrical expression involving the Wasserstein distances for the marginal distributions and the $G$ conditional distributions, namely the dependence of the colatitude distribution on longitude.
\begin{figure}
    \centering
    \includegraphics[width=75mm]{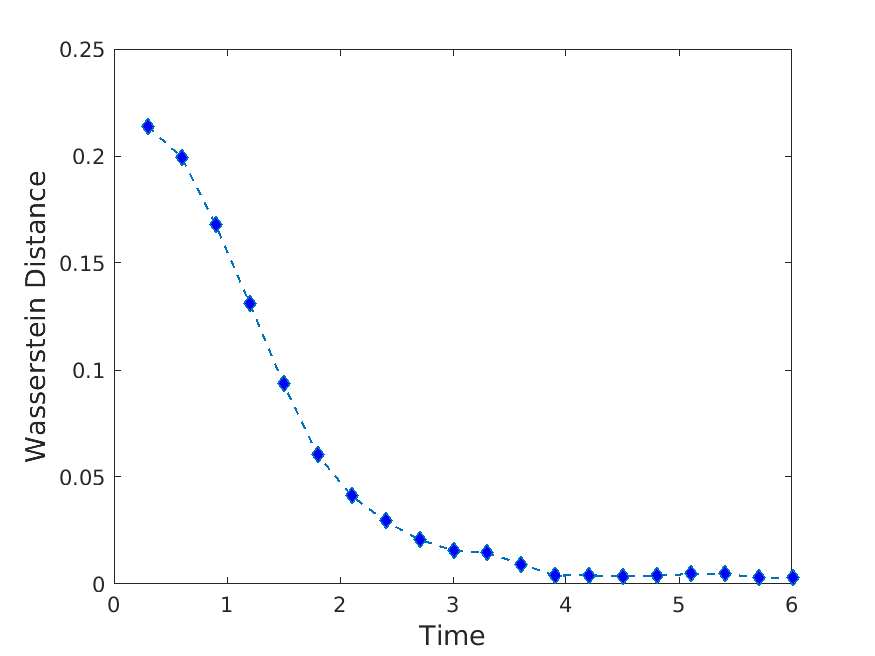}
    \includegraphics[width=75mm]{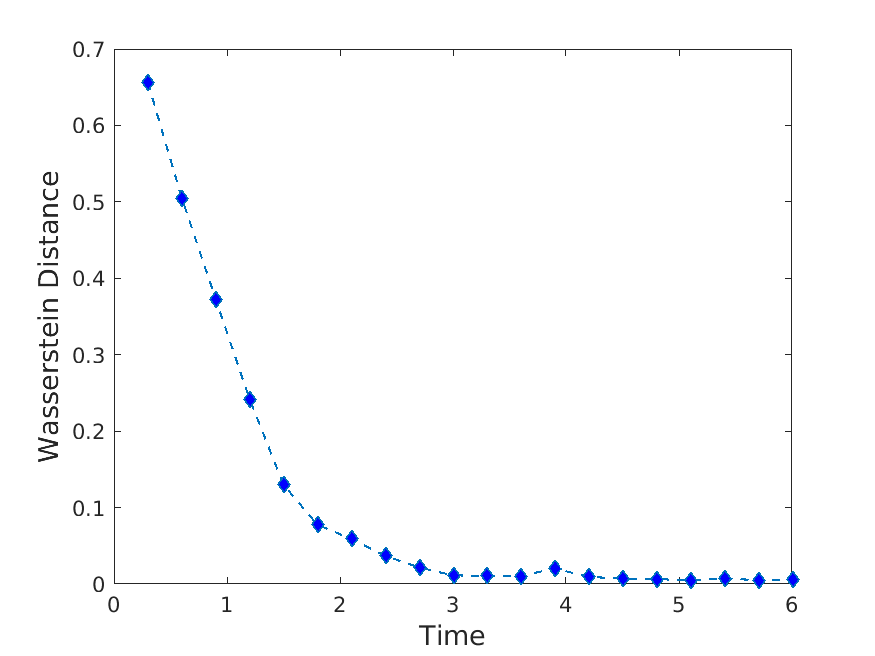}
    \caption{Wasserstein distance between $F_1(\theta )=(\theta+\pi)/(2\pi )$ and $F_{N}^{\theta_s}$ (\textbf{left}) and Wasserstein distance between $G_1(\phi)$ which has $g_1(\phi)=1/2$ as pdf, and $G_{N}^{\phi_s}$  
     (\textbf{right}) with $N = 10^5$ and $s$ ranging over $0.3, 0.6,0.9,\dots, 6.0$.}
    \label{fig:wass_dist}
\end{figure}

For each $s \in [0,10]$, let $F^{\theta_s}$ and $G^{\phi_s}$ be the marginal CDFs of $\theta_s$ and $\phi_s$ respectively. For $N \in \mathbb N$, denote by $F_{N}^{\theta_s}$ and $G_{N}^{\phi_s}$ the empirical CDFs of $\theta_s$ and $\phi_s$. 
We generate empirical CDFs $F_{N}^{\theta_s}$ and $G_{N}^{\phi_s}$ with $s=0.3, 0.6,0.9, \dots, 6.0$, and $N=10^5$. Figure~\ref{fig:wass_dist} demonstrates that $W_1(F_1, F_{N}^{\theta_s})$ and  $W_1(G_1, G_{N}^{\phi_s})$, each decreases with increasing $s$. As a consequence of Theorem~\ref{sanov}~and Proposition \ref{prop:fluctuation} for $N=10^5$ with probability at least $0.99$ it holds that $W_1(F_{N}^{\theta_s},F^{\theta_s}) \leq 0.025$ and $W_1(G_{N}^{\phi_s}, G^{\phi_s}) \leq 0.018$. Thus, we observe that $F^{\theta_s}$ and $G^{\phi_s}$ converge to $F_1$ and $G_1$ respectively. 
\vskip.1in

\paragraph{\textbf{Hypothesis tests for independence and goodness-of-fit.}} We run a total of $22$ hypothesis tests to examine the evolution of the joint distribution of the angles  $\theta_s$ and $\phi_s$.
In order to account for multiple testing, we set the significance level of each test to $0.00045$, leading to an overall level of $0.01$.
First, we generate sample paths to obtain $N =10^5$ realisations 
of $(\theta_{s},\phi_{s})$ for each value of $s = 0.3, 0.6, 0.9,\dots, 6.0$. For each $s$ we test the null hypothesis $H_{0,s}$ that the angles $\theta_s$ and $\phi_s$ are statistically independent, against the alternative hypothesis $H_{1,s}$ that they are dependent. To this end, we rely on a widely used nonparametric independence test, which is based on the Hilbert-Schmidt Independence Criterion (HSIC) dependence measure \cite{GG, adaptive}; the implementation is due to \cite{JZG}. 
It is observed that while the null hypothesis is rejected for $s=0.3, \dots, 2.1$, the test is unable to reject $H_{0,s}$ from $s=2.4,\dots,6.0$ at (an overall) significance level $0.01$. For $s=10$, we run two Kolmogorov-Smirnov goodness-of-tests as follows. The first one tests the null hypothesis $H_{0}^{\theta_s}$ that $\theta_s,~s=10$ is distributed according to $F_{1}$ against the alternative that it is not; the second one tests the null hypothesis $H_{0}^{\phi_s}$ that $\phi_s,~s=10$ is distributed according to $G_{1}$ against the alternative that it is not. 
At significance level $0.01$, the test is unable to reject the null hypotheses $H_{0}^{\theta_s}$ and $H_{0}^{\phi_s}$ for $s=10$.

\section{Acknowledgements} The authors thank Nadia Mazza for helpful remarks on combinatorics. MKS is funded by a Faculty of Science and Technology studentship, Lancaster University.

\end{document}